\documentclass[10pt,french]{smfart}

\usepackage[T1]{fontenc}
\usepackage[english,francais]{babel}
\usepackage{latexsym,amscd,color}
\usepackage{amsmath,amsfonts,amssymb,mathrsfs}
\usepackage{enumerate,euscript}
\usepackage{amssymb,url,xspace,smfthm}
\input xy
\xyoption{all}

\newcommand{\BibTeX}{{\scshape Bib}\kern-.08em\TeX}
\newcommand{\T}{\S\kern .15em\relax }
\newcommand{\AMS}{$\mathcal{A}$\kern-.1667em\lower.5ex\hbox
        {$\mathcal{M}$}\kern-.125em$\mathcal{S}$}

\DeclareMathOperator{\des}{des}

\DeclareMathOperator{\ndeg}{\widehat{\deg}_{\mathrm{n}}}
\DeclareMathOperator{\npdeg}{\widehat{\deg}_{\mathrm{n+}}}

\DeclareMathOperator{\hdeg}{\widehat{\deg}}
\DeclareMathOperator{\hmu}{\widehat{\mu}}

\DeclareMathOperator{\Hom}{Hom}

\DeclareMathOperator{\Image}{Im}

\DeclareMathOperator{\rang}{rg}

\DeclareMathOperator{\sq}{sq}

\DeclareMathOperator{\Spec}{Spec}
\DeclareMathOperator{\vol}{vol}
\DeclareMathOperator{\hvol}{\widehat{\vol}}
\DeclareMathOperator{\hnvol}{\widehat{\vol}_{\mathrm{n}}}

\def\sbullet{{\scriptscriptstyle\bullet}}

\renewcommand\det{\mathrm{d\acute{e}t}}
\def\huayi{}

\title[Majorations explicites des fonctions de Hilbert-Samuel]{Majorations explicites des fonctions de Hilbert-Samuel g\'eom\'etrique et arithm\'etique}
\date{\today}
\author{Huayi Chen}
\address{Universit\'e Grenoble Alpes, Institut Fourier (UMR 5582), F-38402 Grnoble, France}
\email{huayi.chen@ujf-grenoble.fr}
\urladdr{}

\begin{document}
\def\smfbyname{}

\begin{abstract}
En utilisant l'approche de $\mathbb R$-filtration en g\'eom\'etrie d'Arakelov, on \'etablit des majorations explicites des fonctions de Hilbert-Samuel g\'eom\'etrique et arithm\'etique pour les fibr\'es inversibles sur une vari\'et\'e projective et les fibr\'es inversibles hermitiens sur une vari\'et\'e projective arithm\'etique.
\end{abstract}

\begin{altabstract}
By using the $\mathbb R$-filtration approach of Arakelov geometry, one establishes explicit upper bounds for geometric and arithmetic Hilbert-Samuel function for line bundles on projective varieties and hermitian line bundles on arithmetic projective varieties. 
\end{altabstract}
\maketitle

\tableofcontents
\newpage
\section{Introduction}

Soient $X$ un sch\'ema projectif et int\`egre d\'efini sur un corps $k$, et $L$ un $\mathcal O_X$-module inversible. Rappelons que la fonction de Hilbert-Samuel de $L$ est d\'efinie comme l'application de $\mathbb N$ vers $\mathbb N$ qui envoie $n\in\mathbb N$ en $h^0(L^{\otimes n})$, le rang de l'espace des sections globales $H^0(X,L^{\otimes n})$ sur $k$. Le th\'eor\`eme de Riemann-Roch et le th\'eor\`eme d'annihilation de Serre montrent que, si le faisceau $L$ est ample, alors la relation suivante est v\'erifi\'ee~:
\begin{equation}\label{Equ:RRH}h^0(L^{\otimes n})=c_1(L)^d\frac{n^d}{d!}+o(n^d),\end{equation}
o\`u $d$ est la dimension de Krull du sch\'ema $X$, et $c_1(L)^d$ est le nombre d'auto-intersection de $L$. Le th\'eor\`eme d'approximation de Fujita \cite{Fujita94,Takagi07} montre que la relation \eqref{Equ:RRH} est v\'erifi\'ee en g\'en\'erale, quitte \`a remplacer le nombre d'intersection $c_1(L)^d$ par le volume de $L$, d\'efini comme 
\[\mathrm{vol}(L):=\limsup_{n\rightarrow+\infty}\frac{h^0(L^{\otimes n})}{n^d/d!}.\]
En d'autres termes, la limite sup\'erieure d\'efinissant la fonction volume est en fait une limite.

Il est cependant plus d\'elicat d'\'etudier l'estimation explicite de la fonction de Hilbert-Samuel qui sont valables pour tout entier $n\geqslant 1$. Dans la litt\'erature, diff\'erentes approches ont \'et\'e propos\'ees, souvent sous des conditions de positivit\'e forte pour le faisceau inversible $L$. On peut consulter par exemple les travaux de Nesterenko \cite{Nesterenko85}, Chardin \cite{Chardin89} et Sombra \cite{Sombra97}, o\`u on suppose que le faisceau $L$ est tr\`es ample et fixe un plongement de la vari\'et\'e polaris\'ee $(X,L)$ dans un espace projectif. Une majoration explicite de $h^0(L^{\otimes n})$ est ensuite obtenue par r\'ecurrence sur la dimension $d$ du sch\'ema $X$, en utilisant l'intersection avec des hyperplanes de l'espace projectif. Cette approche a une nature alg\'ebrique car le choix d'un plongement de la vari\'et\'e polaris\'ee correspond \`a un syst\`eme de g\'en\'erateurs homog\`enes de l'alg\`ebre gradu\'ee des sections globales des puissances tensorielles de $L$. On renvoie les lecteurs vers l'article de Bertrand dans \cite[chapitre 9]{LNM1752} pour une pr\'esentation d\'etaill\'ee de cette m\'ethode. L'approche de Koll\'ar et Matsusaka \cite{Kollar_Matsusaka} repose sur la comparaison entre la fonction $h^0$ et la caract\'eristique d'Euler-Poincar\'e (somme altern\'ee des rangs des espaces de cohomologie). Cette m\'ethode est relativement plus proche de l'esprit du th\'eor\`eme de Riemann-Roch. On suppose que le sch\'ema $X$ est r\'egulier et que le faisceau $L$ est semi-ample, i.e., une puissance tensorielle de $L$ est sans lieu de base. Un encadrement effectif mais assez compliqu\'e dans le cas de dimension sup\'erieure a \'et\'e obtenu pour $h^0(L^{\otimes n})$. L'encadrement ne d\'epend que du nombre d'auto-intersection $c_1(L)^d$ et le nombre d'intersection $c_1(L)^{d-1}c_1(\omega_X)$, o\`u $\omega_X$ est le fibr\'e canonique de $X$.

La fonction de Hilbert-Samuel peut \^etre g\'en\'eralis\'ee dans le cadre de la g\'eom\'etrie d'Arakelov, o\`u on consid\`ere une vari\'et\'e arithm\'etique projective $\mathscr X$ (i.e., un sch\'ema int\`egre, projectif et plat sur $\Spec\mathbb Z$) et un faisceau inversible $\mathscr L$ sur $\mathscr X$ muni d'une m\'etrique continue sur $\mathscr L(\mathbb C)$, invariante par la conjugaison complexe (ces donn\'ees sont appel\'ees un \emph{faisceau inversible hermitien} sur $\mathscr X$ et not\'ees comme $\overline{\mathscr L}$). Similairement \`a la situation g\'eom\'etrique, on d\'efinit la fonction de Hilbert-Samuel arithm\'etique de $\overline{\mathscr L}$ comme la fonction de $\mathbb N$ vers $[0,+\infty[$ qui envoie $n$ en $\widehat{h}^0(\overline{\mathscr L}{}^{\otimes n})$, le logarithme du nombre des sections globales de $\mathscr L^{\otimes n}$ dont la norme sup est major\'ee par $1$. La suite 
\[\frac{\widehat{h}^0(\overline{\mathscr L}{}^{\otimes n})}{n^{d+1}/(d+1)!},\quad n\geqslant 1\]
poss\`ede une limite\footnote{La convergence de cette suite a \'et\'e d\'emontr\'ee dans \cite{Chen08}. On peut aussi la d\'eduire du th\'eor\`eme de Fujita arithm\'etique \cite{Chen10,Yuan09}.} que l'on note comme $\widehat{\vol}(\overline{\mathscr L})$, et on l'appelle le \emph{volume arithm\'etique} de $\overline{\mathscr L}$. Ici $d$ d\'esigne la dimension relative de $\mathscr X\rightarrow\Spec\mathbb Z$ (et donc la dimension de $\mathscr X$ est $d+1$). On peut aussi exprimer ce r\'esultat comme une formule asymptotique 
\begin{equation}
\widehat{h}^0(\overline{\mathscr L}{}^{\otimes n})=\widehat{\vol}(\overline{\mathscr L})\frac{n^{d+1}}{(d+1)!}+o(n^{d+1}).
\end{equation}

Du point de vue birationnel, il est naturel de se demander si on peut obtenir une estimation de  $h^0(L^{\otimes n})$ en fonction des invariants birationnels de $L$ (comme par exemple le volume de $L$). La m\^eme question se pose aussi pour la fonction $\widehat{h}^0$ dans le cadre arithm\'etique. Cependant, les outils que l'on dispose, comme par exemple le th\'eor\`eme d'approximation de Fujita (g\'eom\'etrique ou arithm\'etique) sous forme actuelle, ne permet pas de traiter ce probl\`eme de fa\c{c}on effective. Il est encore peu probable que les m\'ethodes que l'on a r\'esum\'e plus haut se g\'en\'eralisent dans la situation birationnelle ou s'adaptent facilement dans le cadre de la g\'eom\'etrie arithm\'etique. Les r\'esultats arithm\'etiques sont rares dans la litt\'erature et portent notamment sur les cas o\`u la dimension de la vari\'et\'e arithm\'etique est petite. On peut consulter par exemple les r\'esultats de Blichfeld \cite{Blichfeld14}, Henk \cite{Henk02} et Gaudron \cite{Gaudron09} pour le cas d'une courbe arithm\'etique (ces r\'esultats sont bas\'es sur la g\'eom\'etrie des nombres) et le travail de Yuan et Zhang \cite{Yuan_Zhang13} pour le cas d'une surface arithm\'etique.

Le but principal de cet article est d'\'etablir une majoration effective pour la fonction de Hilbert-Samuel arithm\'etique en dimension quelconque. Dans le cas de surface arithm\'etique, ce r\'esultat est asymptotiquement plus pr\'ecis que la majoration obtenue dans \cite{Yuan_Zhang13} (cf. la remarque \ref{Rem:comparaisonYuanZhang} {\it infra.}).

\begin{theo}\label{Thm:majorationgeom}
Soit $\mathscr X$ une vari\'et\'e arithm\'etique projective d\'efinie sur l'anneau des entiers alg\'ebriques d'un corps de nombres $K$. Il existe une application $\widehat{\varepsilon}$ que l'on explicitera, de l'ensemble des faisceaux inversibles hermitiens  gros sur $\mathscr X$ vers $[0,+\infty[$, qui v\'erifie les conditions suivantes~:
\begin{enumerate}[(a)]
\item si $\overline{\mathscr L}$ et $\overline{\mathscr L}{}'$ sont deux faisceaux inversibles hermitiens gros tels que $\overline{\mathscr L}{}^\vee\otimes\overline{\mathscr L}{}'$ poss\`ede au moins une section effective non-nulle, alors $\widehat{\varepsilon}(\overline{\mathscr L})\leqslant\widehat{\varepsilon}(\overline{\mathscr L}{}')$;
\item pour tout faisceau inversible hermitien gros $\overline{\mathscr L}$ sur $\mathscr X$, on a \[\widehat{h}^0(\overline{\mathscr L})\leqslant\frac{\hvol(\overline{\mathscr L})}{(d+1)!}+\widehat{\varepsilon}(\overline{\mathscr L}),\]
o\`u $d+1$ est la dimension de Krull du sch\'ema $\mathscr X$;
\item pour tout faisceau inversible hermitien gros $\overline{\mathscr L}$, on a
\[\widehat{\varepsilon}(\overline{\mathscr L}{}^{\otimes n})= [K:\mathbb Q]\frac{\vol(\mathscr L_K)}{(d-1)!}n^d\ln(n)+O(n^d).\]
\end{enumerate} 
\end{theo}
{\huayi Le terme d'erreur $\widehat{\varepsilon}(\overline{\mathscr L})$ sera rendu explicite plus loin dans \S\,\ref{Sec:majroationarith} (cf. le th\'eor\`eme \ref{Thm:majorationh0}), il d\'epend de la pente maximale asymptotique de $\overline{\mathscr L}$ (qui est un invariant arithm\'etique birationnel), ainsi que des invariants birationnels de $\mathscr L_K$. Si on compare ce th\'eor\`eme aux r\'esultats dans la litt\'erature, il y a deux nouveaut\'es essentielles. Premi\`erement l'in\'egalit\'e dans (b) est valable pour le fibr\'e inversible hermitien $\overline{\mathscr L}$. On n'a pas besoin de passer \`a une puissance tensorielle d'exposant suffisamment grand, comme par exemple dans le th\'eor\`eme de Hilbert-Samuel arithm\'etique de \cite{Abbes-Bouche}. Deuxi\`emement, dans le th\'eor\`eme \ref{Thm:majorationgeom} on ne demande aucune condition de positivit\'e sur la m\'etrique du fibr\'e inversible hermitien $\overline{\mathscr L}$.}

La d\'emonstration du th\'eor\`eme repose sur la m\'ethode de $\mathbb R$-filtration introduite dans \cite{Chen10b,Chen10}. On consid\`ere chaque $\mathcal E_n=H^0(\mathscr X,\mathscr L^{\otimes n})$ ($n\in\mathbb N$) comme un r\'eseau dans un $\mathbb R$-espace vectoriel muni de la norme sup. Les minima successifs du r\'eseau correspondent \`a une $\mathbb R$-filtration $\mathcal F$ sur le $\mathbb Q$-espace vectoriel $E_n=H^0(\mathscr X_K,\mathscr L_K)$ telle que
\[\forall\,t\in\mathbb R,\quad \mathcal F^t(E_n)=\mathrm{Vect}\big\{s\in\mathcal E_n\,:\,\|s\|_{\sup}\leqslant\mathrm{e}^{-t}\big\}.\]
Rappelons que, pour tout $i\in\{1,\ldots,R_n\}$, o\`u $R_n=\rang_{\mathbb Q}(E_n)$, le $i^{\text{\`eme}}$ minimum logarithmique du r\'eseau $\mathcal E_n$ est d\'efini comme \[\lambda_i(\mathcal E_n):=\sup\{t\in\mathbb R\,:\,\rang_{\mathbb Q}(\mathcal F^t(E_n))\geqslant i\}.\]
En particulier, on peut interpr\'eter $\sum_{i}\max(\lambda_i(\mathcal E_1),0)$ sous forme d'une int\'egrale~: 
\[\sum_{i=1}^{R_1}\max(\lambda_i(\mathcal E_1),0)=\int_{0}^{+\infty}\rang(\mathcal F^tE_1)\,\mathrm{d}t.\]
Il s'av\`ere que cette somme est \'etroitement li\'ee \`a $\widehat{h}^0(\overline{\mathscr L}{}^{\otimes n})$, compte tenu du deuxi\`eme th\'eor\`eme de Minkowski et des r\'esultats en g\'eom\'etrie des nombres comme par exemple \cite{Blichfeld14}. En outre, pour tout $t\in\mathbb R$, la somme directe $E^t_\sbullet=\bigoplus_{n\geqslant 0}\mathcal F^{nt}(E_n)$ est un syst\`eme lin\'eaire gradu\'e de $\mathscr L_K$ (o\`u on consid\`ere $\mathscr X_K$ comme un $\mathbb Q$-sch\'ema projectif). Si on d\'esigne par $\vol(E_\sbullet^t)$ son volume, d\'efini comme
\[\vol(E_\sbullet^t)=\limsup_{n\rightarrow+\infty}\frac{\rang_{\mathbb Q}(E_n^t)}{n^d/d!},\]
on peut aussi exprimer le volume arithm\'etique $\hvol(\overline{\mathscr L})$ comme une int\'egrale
\[\hvol(\overline{\mathscr L})=(d+1)\int_0^{+\infty}\vol(E_\sbullet^t)\,\mathrm{d}t.\]
Ainsi on peut ramener le probl\`eme \`a la majoration de $\rang_{\mathbb Q}(\mathcal F^tE_1)=\rang_{\mathbb Q}(E_1^t)$ en fonction de $\vol(E_\sbullet^t)$. Cela peut \^etre consid\'er\'e comme une g\'en\'eralisation du probl\`eme de majoration explicite de la fonction de Hilbert-Samuel g\'eom\'etrique dans le cadre des syst\`emes lin\'eaires gradu\'es. Ce probl\`eme est r\'esolu par le th\'eor\`eme suivant, qui peut \^etre vu comme un avatar g\'eom\'etrique du th\'eor\`eme \ref{Thm:majorationgeom}.

\begin{theo}\label{Thm:HMgeometryique}
Soit $X$ un sch\'ema projectif et int\`egre d\'efini sur un corps $k$. On d\'esigne par $S$ l'ensemble des syst\`emes lin\'eaires gradu\'es de faisceaux inversibles sur $X$, qui contiennent des diviseurs amples. Il existe une application $\varepsilon:S\rightarrow [0,+\infty[$ que l'on explicitera, qui v\'erifie les conditions suivantes~:
\begin{enumerate}[(a)]
\item si $V_\sbullet$ et $W_\sbullet$ sont des syst\`emes lin\'eaires gradu\'es dans $S$, des faisceaux inversibles $L$ et $M$ respectivement et si le faisceau inversible $L^\vee\otimes M$ admet une section effective non-nulle  $s$ telle que la multiplication par des puissances de $s$ envoie $V_\sbullet$ dans $W_\sbullet$, alors on a $\varepsilon(V_\sbullet)\leqslant\varepsilon(W_\sbullet)$;
\item pour tout syst\`eme lin\'eaire gradu\'e $V_\sbullet$ dans $S$, on a 
\[\rang_k(V_1)\leqslant\frac{\vol(V_\sbullet)}{d!}+\varepsilon(V_\sbullet),\]
o\`u $d$ est la dimension de Krull de $X$;
\item pour tout syst\`eme lin\'eaire gradu\'e $V_\sbullet$ dans $S$ et tout entier $n\geqslant 1$, on a
\[\varepsilon(V_\sbullet^{(n)})\leqslant n^{d-1}\varepsilon(V_\sbullet),\]
o\`u $V_\sbullet^{(n)}=\bigoplus_{m\geqslant 0}V_{nm}$.
\end{enumerate}
\end{theo}

Compar\'e aux r\'esultats dans la litt\'erature, le th\'eor\`eme \ref{Thm:HMgeometryique} s'applique \`a des syst\`emes lin\'eaires gradu\'es tr\`es g\'en\'eraux, et on ne demande pas la condition d'amplitude (ou de semi-amplitude) pour les faisceaux inversibles en question. {\huayi Le terme d'erreur $\varepsilon(.)$ sera pr\'ecis\'e dans \S\,\ref{Sec:HS} et d\'epend du choix d'une cha\^ine de sous-extensions du corps des fonctions rationnelles $k(X)$ sur le corps de base $k$ dont les extensions successives sont de degr\'e de transcendance $1$.}

Pour un syst\`eme lin\'eaire gradu\'e $V_\sbullet$ fix\'e, si on se contente d'obtenir l'existence d'une fonction $F_{V_\sbullet}:\mathbb N\rightarrow+\infty$ telle que  $F_{V_\sbullet}(n)=O(n^{d-1})$ pour $n\rightarrow+\infty$ et que \[\rang_k(V_n)\leqslant \frac{\vol(V_\sbullet)}{d!}n^d+F_{V_\sbullet}(n),\] on peut utiliser la th\'eorie des corps d'Okounkov d\'evelopp\'ee dans \cite{Lazarsfeld_Mustata08,Kaveh_Khovanskii} pour relier $V_\sbullet$ \`a un corps convexe $\Delta$ dans $\mathbb R^d$. On peut majorer le rang de $V_n$ par le nombre de points \`a coordonn\'ees entiers dans $n\Delta$ et ensuite faire appel \`a un r\'esultat de Betke et B\"or\"oczky \cite{Betke_Boroczky99} pour obtenir la majoration asymptotique. Cependant, cette m\'ethode est inad\'equate pour l'application dans la situation arithm\'etique. En effet, pour obtenir une majoration de la fonction de Hilbert-Samuel arithm\'etique, il faut appliquer la majoration de la fonction de Hilbert-Samuel g\'eom\'etrique \`a une famille continue de syst\`emes lin\'eaires gradu\'es. Cependant, le terme sous-dominant dans la majoration de la fonction de Hilbert-Samuel g\'eom\'etrique obtenue par cette m\'ethode d\'epend du bord du corps convexe associ\'e au syst\`eme lin\'eaire gradu\'e. Il est difficile d'obtenir un contr\^ole explicite et uniforme pour la famille de syst\`emes lin\'eaires gradu\'es qui apparaissent dans l'\'etude de la fonction de Hilbert-Samuel arithm\'etique. En outre, cette m\'ethode ne peut pas \^etre directement appliqu\'ee dans la situation arithm\'etique car dans l'analogue arithm\'etique du corps d'Okounkov, il n'y a pas de lien entre l'ensemble des sections effectives du faisceau inversible hermitien avec l'ensemble des points de coordonn\'ees enti\`eres dans le corps d'Okounkov arithm\'etique associ\'e.

Pour d\'emontrer le th\'eor\`eme \ref{Thm:HMgeometryique}, le point cl\'e est d'adopter un point de vue arithm\'etique. En effet, l'approche de $\mathbb R$-filtration s'applique \'egalement dans le cadre de la g\'eom\'etrie arithm\'etique sur le corps de fonctions, o\`u consid\`ere $X$ comme une fibration au-dessus d'une courbe projective r\'eguli\`ere sur $k$. Une telle fibration est toujours r\'ealisable, quitte \`a remplacer $X$ par une modification birationnelle, o\`u la fonction volume reste invariante. On utilise ainsi un argument de nature arithm\'etique comme dans la strat\'egie de d\'emonstration du th\'eor\`eme \ref{Thm:majorationgeom} et ram\`ene le probl\`eme \`a un probl\`eme similaire pour la fibre g\'en\'erique de $X$, qui est un sch\'ema projectif et int\`egre de dimension $\dim(X)-1$ d\'efinie sur le corps de fonction de la courbe de base. La majoration est obtenue par r\'ecurrence sur la dimension de $X$, et le majorant d\'epend du choix d'un tour de fibrations sur courbes d'une modification birationnelle de $X$. Dans le cas o\`u la caract\'eristique de $k$ est z\'ero, on peut utiliser la $\mathbb R$-filtration de Harder-Narasimhan. Cependant, dans le cas o\`u la caract\'eristique de $k$ est positif, il faut utiliser la filtration par minima. Le majorant est l\'eg\`erement plus grand, mais toujours de m\^eme ordre de grandeur.

Cette approche de $\mathbb R$-filtration, qui s'applique \`a la fois aux cas g\'eom\'etrique et arithm\'etique, combine les avantages de plusieurs m\'ethodes mentionn\'ees plus haut. D'abord le majorant de la fonction de Hilbert-Samuel g\'eom\'etrique ou arithm\'etique est obtenu par une formule de r\'ecurrence sur la dimension de $X$, qui rend le calcul explicit. Deuxi\`emement, la contribution arithm\'etique du syst\`eme lin\'eaire gradu\'e par rapport aux courbes projectives r\'eguli\`eres figurant dans le tour de fibrations ressemble beaucoup \`a la contribution du faisceau inversible dualisant dans l'approche de Koll\'ar et Matsusaka.  Enfin, cette m\'ethode peut \^etre naturellement g\'en\'eralis\'ee dans le cadre de syst\`eme lin\'eaire gradu\'e filtr\'e comme dans \cite{Boucksom_Chen}, qui permet de d\'ecouvrir de nouveaux ph\'enom\`enes en g\'eom\'etrie arithm\'etique. On \'etablit par exemple le r\'esultat suivant.

\begin{theo}\label{Thm:majorationdeminima}
Soit $\mathscr X$ une vari\'et\'e arithm\'etique projective d\'efinie sur l'anneau des entiers alg\'ebriques d'un corps de nombres $K$. Il existe une application $\widehat{\epsilon}$ que l'on explicitera, de l'ensemble des faisceaux inversibles hermitiens  gros sur $\mathscr X$ vers $[0,+\infty[$, qui v\'erifie les conditions suivantes~:
\begin{enumerate}[(a)]
\item si $\overline{\mathscr L}$ et $\overline{\mathscr L}{}'$ sont deux faisceaux inversibles hermitiens gros tels que $\overline{\mathscr L}{}^\vee\otimes\overline{\mathscr L}{}'$ poss\`ede au moins une section effective non-nulle, alors $\widehat{\epsilon}(\overline{\mathscr L})\leqslant\widehat{\epsilon}(\overline{\mathscr L}{}')$;
\item pour tout faisceau inversible hermitien gros $\overline{\mathscr L}$ sur $\mathscr X$, on a \[\sum_{i=1}^{r}\max\big(\lambda_i(H^0(\mathscr X,\mathscr L),\|.\|_{\sup}),0\big)\leqslant\frac{\hvol(\overline{\mathscr L})}{(d+1)!}+\widehat{\epsilon}(\overline{\mathscr L}),\]
o\`u $d+1$ est la dimension de Krull du sch\'ema $\mathscr X$, et $r=\rang_{\mathbb Z}H^0(\mathscr X,\mathscr L)$;
\item pour tout faisceau inversible hermitien gros $\overline{\mathscr L}$, on a
$\widehat{\epsilon}(\overline{\mathscr L}{}^{\otimes n})\leqslant n^d\widehat{\epsilon}(\overline{\mathscr L})$.
\end{enumerate} 
\end{theo}

La diff\'erence principale  entre ce th\'eor\`eme et le th\'eor\`eme \ref{Thm:majorationgeom} est dans la condition (c). Au lieu d'avoir un terme d'erreur d'ordre $n^d\ln(n)$, le terme d'erreur $\widehat{\epsilon}(\overline{\mathscr L}{}^{\otimes n})$ ici ({\huayi qui sera pr\'ecis\'e dans la d\'emonstration du th\'eor\`eme \ref{Thm:HSarithmetic}}) est d'ordre $n^d$ lorsque $n\rightarrow+\infty$. {\huayi Un r\'esultat similaire pour les minima successifs absolus est \'etabli dans le th\'eor\`eme \ref{Thm:HMarithmetique}.} Ce r\'esultat est frappant car dans une formule de d\'eveloppement d'une fonction arithm\'etique de type Hilbert-Samuel, on attend souvent que le terme sous-dominant soit d'ordre $O(n^d\ln(n))$ quand $n$ tend vers l'infini. L'estimation (c) dans le th\'eor\`eme \ref{Thm:majorationdeminima} sugg\`ere que le terme d'ordre $O(n^d\ln(n))$ dans le d\'eveloppement de la fonction de Hilbert-Samuel arithm\'etique provient notamment de la comparaison entre diff\'erents invariants arithm\'etiques de fibr\'es vectoriels norm\'es sur la courbe arithm\'etique, ou de la distorsion entre les choix de diff\'erents m\'etriques. La contribution g\'eom\'etrique pourrait agir plut\^ot sur le terme suivant d'ordre $O(n^d)$. On esp\`ere que ce nouveau point de vue nous aidera \`a mieux comprendre le r\^ole de la g\'eom\'etrie du sch\'ema $\mathscr X_K$ dans l'\'etude de la fonction de Hilbert-Samuel arithm\'etique.

Pendant la r\'edaction de l'article, Xinyi Yuan et Tong Zhang m'ont communiqu\'e leur travaux \cite{Yuan_Zhang14}, o\`u ils ont obtenu ind\'ependemment des r\'esultats similaires aux th\'eor\`emes \ref{Thm:majorationgeom} et \ref{Thm:majorationdeminima}. Leur approche a certaines similitudes compar\'ee \`a celle adopt\'ee dans cet article, notamment l'arguement de r\'ecurrence sur la dimension de la vari\'et\'e g\'eom\'etrique ou arithm\'etique. La diff\'erence majeure entre les deux approches repose sur la r\'ealisation du proc\'ed\'e de r\'ecurrence. Dans \cite{Yuan_Zhang14}, l'argument de Yuan et Zhang est bas\'e sur la positivit\'e du fibr\'e inversible et les termes d'erreur dans leurs th\'eor\`emes d\'ependent des nombres d'intersections de certains fibr\'es inversibles auxiliaires qui contr\^olent la positivit\'e du fibr\'e inversible dont on veut borner la fonction de Hilbert-Samuel. Cependant, dans l'article pr\'esent, on choisit de g\'en\'eralizer le probl\`eme dans le cadre des syst\`emes lin\'eaires gradu\'es, munis des structures de m\'etriques et puis ramener le probl\`eme \`a la fibre g\'en\'erique de la vari\'et\'e (arithmetique ou fibr\'ee sur une courbe) afin de r\'eduire la dimension. Il est une question d\'elicate de comparer les termes d'erreur obtenus par ces m\'ethodes diff\'erentes car la liaison entre les deux approches est encore obscure, mais il n'est pas exclu qu'une combinaison astucieuse  de ces m\'ethodes conduira \`a une majoration de la fonction de Hilbert-Samuel g\'eom\'etrique ou arithm\'etique, o\`u le terme d'erreur ne d\'epend que du volume du fibr\'e inversible et le produit d'intersection positif du fibr\'e inversible avec le faisceau dualisant.

L'article est organis\'e comme la suite. Dans le deuxi\`eme paragraphe, on \'etablit un lien entre la valeur maximale du polygone de Harder-Narasimhan d'un fibr\'e vectoriel sur une courbe avec la dimension de l'espace vectoriel des sections globales du fibr\'e vectoriel. Cette comparaison sera utile plus loin dans la majoration de la fonction de Hilbert-Samuel g\'eom\'etrique. Le troisi\`eme paragraphe est consacr\'e \`a un rappel de la notion de pente maximale asymptotique pour les syst\`emes lin\'eaires gradu\'es. C'est un invariant birationnel qui interviendra dans le terme d'erreur de la majoration. Dans le quatri\`eme paragraphe, on propose une nouvelle notion~: tour de fibrations sur courbes, o\`u on consid\`ere une vari\'et\'e projective comme des fibrations successives sur les courbes projectives r\'eguli\`eres d\'efinies sur des corps de plus en plus gros. C'est un outil essentiel pour la majoration de la fonction de Hilbert-Samuel. En utilisant cet outil et la $\mathbb R$-filtration de Harder-Narasimhan, on \'etablit la majoration explicite de la fonction de Hilbert-Samuel g\'eom\'etrique dans le cinqui\`eme paragraphe, sous condition que le corps de base est de caract\'eristique z\'ero. Le sixim\`eme paragraphe est consacr\'e \`a un rappel sur la notion de fibr\'e vectoriel ad\'elique sur un corps de nombres, et la filtration par minima absolus. On obtient dans le septi\`eme paragraphe la majoration explicite de la fonction de Hilbert-Samuel arithm\'etique. Enfin, dans le dernier paragraphe, on d\'emontre la majoration de la fonction de Hilbert-Samuel g\'eom\'etrique dans la cas de caract\'eristique positif, en utilisant la m\'ethode arithm\'etique en consid\'erant la $\mathbb R$-filtration par minima.
\vspace{2mm}

\noindent{\bf Remerciements~:} Je voudrais remercier \'Eric Gaudron pour des remarques qui m'ont aid\'e \`a am\'eliorer la r\'edaction de l'article. 
Pendant la pr\'eparation et la r\'edaction de l'article, j'ai  b\'en\'efici\'e des discussions avec Sebastien Boucksom, je tiens \`a lui exprimer mes gratitudes.
Enfin, je suis reconnaissant \`a Xinyi Yuan et Tong Zhang pour m'avoir communiqu\'e leur article et pour des discussions tr\`es int\'eressantes.  

\section{Degr\'e positif d'un fibr\'e vectoriel}

Soient $k$ un corps et $C$ une courbe projective r\'eguli\`ere d\'efinie sur $k$. Rappelons que la formule de Riemann-Roch montre que, pour tout fibr\'e vectoriel $E$ sur $C$, on a 
\begin{equation}\label{Equ:RR}
h^0(E)-h^1(E)=\deg(E)+\rang(E)(1-g),
\end{equation}
o\`u $h^0(E)$ et $h^1(E)$ sont respectivement la dimension sur le corps $k$ des espaces de cohomologie $H^0(X,E)$ et $H^1(X,E)$, et $g$ d\'esigne le genre de $C$. En utilisant cette formule, on peut relier $h^0(E)$ \`a la valeur maximale du polygone de Harder-Narasimhan de $E$.

\'Etant donn\'e un fibr\'e vectoriel non-nul $E$ sur $C$, la \emph{pente} de $E$ est d\'efinie comme le quotient du degr\'e de $E$ par son rang, not\'ee comme $\mu(E)$. Le fibr\'e vectoriel $E$ est dit \emph{semi-stable} si chaque sous-fibr\'e vectoriel non-nul de $E$ admet une pente $\leqslant\mu(E)$. Si $E$ est un fibr\'e vectoriel non-nul qui n'est pas n\'ecessairement semi-stable, il existe un unique sous-fibr\'e vectoriel $E_{\des}$ de $E$ qui v\'erifie les conditions suivantes
\begin{enumerate}[(a)]
\item pour tout sous-fibr\'e vectoriel non-nul $F$ de $E$, on a $\mu(F)\leqslant\mu(E_{\des})$;
\item si $F$ est un sous-fibr\'e vectoriel non-nul de  $E$ tel que $\mu(F)=\mu(E_{\des})$, alors $F\subset E_{\des}$.
\end{enumerate} Le sous-fibr\'e vectoriel $E_{\des}$ est appel\'e le \emph{sous-fibr\'e d\'estabilisant} de $E$. Sa pente est appel\'ee la \emph{pente maximale} de $E$, not\'ee comme $\mu_{\max}(E)$. 

La condition (b) plus haut implique que le quotient $E/E_{\des}$ est sans torsion, donc est un fibr\'e vectoriel sur $C$. Ainsi on peut construire par r\'ecurrence un drapeau de sous-fibr\'es vectoriels de $E$~:
\[0=E_0\subsetneq E_1\subsetneq \ldots\subsetneq E_n=E\]
tel que $E_i/E_{i-1}=(E/E_{i-1})_{\des}$ pour tout $i\in\{1,\ldots,n\}$. Ce drapeau est appel\'e le  \emph{drapeau de Harder-Narasimhan} de $E$. Chaque sous-quotient $E_i/E_{i-1}$ est un fibr\'e vectoriel semi-stable sur $C$. En outre, si on d\'esigne par  $\alpha_i$ la pente du sous-quotient $E_i/E_{i-1}$, alors les in\'egalit\'es $\alpha_1>\ldots>\alpha_n$ sont v\'erifi\'ees. Il s'av\`ere que le drapeau de Harder-Narasimhan est le seul\label{page:HNunicite} drapeau de sous-fibr\'es vectoriels de $E$ tel que les sous-quotients soient semi-stables et de pentes strictement d\'ecroissantes (cf. \cite[th\'eor\`eme 1.3.4]{Huyb} pour une d\'emonstration). La derni\`ere pente $\alpha_n$ est appel\'ee la \emph{pente minimale} de $E$, not\'ee comme $\mu_{\min}(E)$. C'est aussi la valeur minimale des pentes des fibr\'es vectoriels quotients de $E$. En particulier, les pentes maximale et minimale sont reli\'ees par la formule de dualit\'e suivante~: pour tout fibr\'e vectoriel non-nul $E$ sur $C$, on a 
\[\mu_{\max}(E)+\mu_{\min}(E^\vee)=0.\] 
On d\'esigne par $P_E$ la fonction concave et affine par morceau d\'efinie sur l'intervalle $[0,\rang(E)]$, qui est affine sur chaque intervalle $[\rang(E_{i-1}),\rang(E_i)]$ et de pente $\alpha_i$. Rappelons que le graphe de $P_E$ s'identifie au bord sup\'erieur de l'enveloppe convexe de l'ensemble des points de la forme  $(\rang(F),\deg(F))\in\mathbb R^2$, o\`u $F$ parcourt l'ensemble des sous-fibr\'es vectoriels de $E$. La fonction $P_E$ est appel\'ee le \emph{polygone de Harder-Narasimhan} de $E$.

\begin{defi}
Soit $E$ un fibr\'e vectoriel non-nul sur la courbe projective $C$. On d\'esigne par $\deg_+(E)$ la valeur maximale de la fonction $P_E$ sur l'intervalle $[0,\rang(E)]$, appel\'ee le \emph{degr\'e positif} of $E$. On voit aussit\^ot de la d\'efinition que, si la pente minimale de $E$ est positive, alors $\deg_+(E)$ s'identifie au degr\'e de $E$. On convient que le degr\'e positif du fibr\'e vectoriel nul est z\'ero.
\end{defi}

\begin{lemm}\label{Lem:h0} Soient $C$ une courbe projective r\'eguli\`ere de genre $g$ d\'efinie sur un corps $k$, et $E$ un fibr\'e vectoriel non-nul sur $C$.
\begin{enumerate}[(a)]
\item Si $\mu_{\max}(E)<0$, alors $h^0(E)=0$.
\item Si $\mu_{\min}(E)>2g-2$, alors $h^0(E)=\deg(E)+\rang(E)(1-g)$.
\item Si $\mu_{\min}(E)>0$, alors $|h^0(E)-\deg(E)|\leqslant \rang(E)|g-1|$.
\end{enumerate}
\end{lemm}
\begin{proof}
(a) Supposons que $E$ poss\`ede une section globale non-nul. Elle correspond \`a un homomorphisme non-nul de  $\mathcal O_C$ vers $E$. Donc on a \[0=\mu(\mathcal O_C)\leqslant\mu_{\max}(E).\]

(b) D'apr\`es la formule de Riemann-Roch  \eqref{Equ:RR} et la dualit\'e de Serre  $h^1(E)=h^0(E^\vee\otimes\omega_C)$, o\`u $\omega_C$ est le faisceau dualisant sur $C$, on obtient
\[h^0(E)-h^0(E^\vee\otimes\omega_C)=\deg(E)+\rang(E)(1-g).\]
Comme \[\mu_{\max}(E^\vee\otimes\omega_C)=\mu_{\max}(E^\vee)+\deg(\omega_C)=2g-2-\mu_{\min}(E),\]
si $\mu_{\min}(E)>2g-2$, alors on a $\mu_{\max}(E^\vee\otimes\omega_C)<0$. Donc $h^0(E^\vee\otimes\omega_C)=0$ compte tenu de (a). Par cons\'equent, l'\'egalit\'e $h^0(E)=\deg(E)-\rang(E)(1-g)$ est v\'erifi\'ee.

(c) D'apr\`es (b), l'in\'egalit\'e est v\'erifi\'ee lorsque $g\leqslant 1$. Dans la suite, on suppose  $g\geqslant 2$. Comme $\mu_{\min}(E)>0$, on a $\mu_{\min}(E\otimes\omega_C)=\mu_{\min}(E)+2g-2>2g-2$. D'apr\`es (b), on obtient
\[h^0(E\otimes\omega_C)=\deg(E\otimes\omega_C)+\rang(E)(1-g)=\deg(E)+\rang(E)(g-1).\] 
Comme $h^0(
\omega_C)>0$, on a
\[h^0(E)\leqslant h^0(E\otimes\omega_C)\leqslant\deg(E)+\rang(E)(g-1).\]
En outre, d'apr\`es la formule de Riemann-Roch \eqref{Equ:RR}, on a $h^0(E)\geqslant \deg(E)+\rang(E)(1-g)$. Donc l'in\'egalit\'e est d\'emontr\'ee.
\end{proof}

\begin{lemm}\label{Lem:majh0ss}
Soient $C$ une courbe projective r\'eguli\`ere d\'efinie sur un corps $k$, et $E$ un fibr\'e vectoriel non-nul sur $C$.
 Si $E$ est semi-stable et de pente $0$, alors $h^0(E)\leqslant\rang(E)$.
\end{lemm}
\begin{proof}
On peut supposer que $E$ poss\`ede une section globale non-nulle, sinon le r\'esultat est trivial. Cette section d\'efinit un homomorphisme non-nul de $\mathcal O_C$ vers $E$. Comme $E$ est semi-stable de pente $0$, le faisceau quotient $E/\mathcal O_C$ est ou bien nul, ou bien un fibr\'e vectoriel semi-stable de pente $0$. De plus, la suite exacte longue de groupes de cohomologie associ\'ee \`a la suite exacte courte $0\rightarrow\mathcal O_C\rightarrow E\rightarrow E/\mathcal O_C\rightarrow 0$ montre que 
\[h^0(E)\leqslant h^0(\mathcal O_C)+h^0(E/\mathcal O_C)=h^0(E/\mathcal O_C)+1.\] 
Par r\'ecurrence sur le rang de $E$, on obtient le r\'esultat.
\end{proof}

\begin{theo}\label{Thm:h0etdegplus} Soient $C$ une courbe projective r\'eguli\`ere de genre $g$ d\'efinie sur un corps $k$, et $E$ un fibr\'e vectoriel non-nul sur $C$. On a
\begin{equation}\label{Equ:comparison}
|{h}^0(E)-\deg_+(E)|\leqslant \rang(E)\max(g-1,1).
\end{equation}
\end{theo}
\begin{proof}
Soit \[0=E_0\subsetneq E_1\subsetneq\ldots\subsetneq E_n=E\] le drapeau de Harder-Narasimhan de $E$. Pour tout $i\in\{1,\ldots,n\}$, soit $\alpha_i$ la pente de $E_{i}/E_{i-1}$. Soit $j$ le plus grand indice tel que $\alpha_j\geqslant 0$. Si un tel indice n'existe pas, on note $j=0$ par convention. On a $\deg_+(E)=\deg(E_j)$ par d\'efinition. En outre, comme $E/E_j$ est ou bien nul ou bien de pente maximale strictement n\'egative, on a $h^0(E/E_j)=0$ et donc $h^0(E)=h^0(E_j)$. Si $j=0$, alors on a $h^0(E)=0=\deg_+(E)$ et l'in\'egalit\'e devient triviale. Dans la suite, on suppose  $j\in\{1,\ldots,n\}$.

On traite d'abord le cas o\`u $g\geqslant 1$. Si $\alpha_j=\mu_{\min}(E_j)>0$, d'apr\`es le lemme \ref{Lem:h0}.(c), on obtient \begin{equation*}|h^0(E)-\deg_+(E)|=|h^0(E_j)-\deg(E_j)|\leqslant\rang(E_j)(g-1),\end{equation*} 
qui implique \eqref{Equ:comparison}. Il reste le cas o\`u $\alpha_j=0$. Dans ce cas-l\`a $E_j/E_{j-1}$ est un fibr\'e vectoriel semi-stable de pente $0$. D'apr\`es le lemme \ref{Lem:majh0ss} on a $h^0(E_{j}/E_{j-1})\leqslant\rang(E_{j}/E_{j-1})$, qui implique que
\begin{equation}\begin{split}\label{Equ:encardh0}
h^0(E_{j-1})\leqslant h^0(E)=h^0(E_j)&\leqslant h^0(E_{j-1})+h^0(E_j/E_{j-1})\\
&\leqslant
h^0(E_{j-1})+\rang(E_j/E_{j-1}).\end{split}\end{equation}
En outre, on a $\deg_+(E)=\deg(E_{j-1})$. Le fibr\'e vectoriel $E_{j-1}$ est ou bien nul, ou bien de pente minimale $>0$. D'apr\`es le lemme \ref{Lem:h0}.(c) on obtient (l'in\'egalit\'e est triviale lorsque $E_{j-1}=0$)
\[|h^0(E_{j-1})-\deg_+(E)|=|h^0(E_{j-1})-\deg(E_{j-1})|\leqslant \rang(E_{j-1})(g-1).\]
Si on combine cette in\'egalit\'e avec  \eqref{Equ:encardh0}, on obtient \eqref{Equ:comparison}.

Dans la suite, on suppose que $g=0$. Comme $\alpha_j=\mu_{\min}(E_j)\geqslant 0>2g-2$, d'apr\`es le lemme \ref{Lem:h0}.(b) on obtient $h^0(E_j)-\deg(E_j)=\rang(E_j)$. Comme on a $ h^0(E_j)=h^0(E)$ et $\deg(E_j)=\deg_+(E)$, le r\'esultat est aussi vrai dans ce cas-l\`a.
\end{proof}

\`A l'aide de $\mathbb R$-filtration de Harder-Narasimhan introduite dans \cite{Chen10b}, on peut interpr\'eter la fonction $\deg_+(.)$ comme une int\'egrale. Soit $E$ un fibr\'e vectoriel non-nul sur $C$. On suppose que son drapeau de Harder-Narasimhan est
\[0=E_0\subsetneq E_1\subsetneq \ldots\subsetneq E_n=E.\]
Pour tout $i\in\{1,\ldots,n\}$, soit $\alpha_i$ la pente du sous-quotient $E_i/E_{i-1}$. On d\'efinit une famille $(\mathcal F^tE)_{t\in\mathbb R}$ de sous-fibr\'es vectoriels de $E$ comme
\[\mathcal F^tE=E_i\;\text{si $\alpha_i\geqslant t>\alpha_{i-1}$},\]
o\`u par convention $\alpha_{0}=+\infty$ et $\alpha_{n+1}=-\infty$. Pour tout nombre r\'eel $t$, on d\'esigne par $\mathcal F^{t+}E$ le sous-fibr\'e vectoriel $\sum_{a>0}\mathcal F^{t+a}E$ de $E$ et on d\'efinit $\sq^t(E)$ le quotient $\mathcal F^tE/\mathcal F^{t+}E$, appel\'e le \emph{sous-quotient} d'indice $t$ de la filtration $\mathcal F$. D'apr\`es la d\'efinition du drapeau de Harder-Narasimhan, on obtient que chaque sous-quotient $\mathrm{sq}^t(E)$ est ou bien le fibr\'e vectoriel nul, ou bien un fibr\'e vectoriel semi-stable de pente $t$. En outre, l'unicit\'e du drapeau de Harder-Narasimhan que l'on a mentionn\'ee dans la page \pageref{page:HNunicite} conduit au crit\`ere suivant de la $\mathbb R$-filtration de Harder-Narasimhan.

\begin{prop}\label{Pro:criteredeHN}
Soit $E$ un fibr\'e vectoriel non-nul sur $C$ et $(\mathcal G^tE)_{t\in\mathbb R}$ une $\mathbb R$-filtration d\'ecroissante en sous-fibr\'es vectoriels de $E$ telle que $\mathcal G^tE=0$ pour $t$ suffisamment positif, $\mathcal G^tE=E$ pour $t$ suffisamment n\'egatif et $\bigcap_{a>0}\mathcal G^{t-a}E=\mathcal G^tE$ pour tout $t\in\mathbb R$. Alors $\mathcal G$ est la $\mathbb R$-filtration de Harder-Narasimhan si et seulement si, pour tout $t\in\mathbb R$, le sous-quotient d'indice $t$ de la filtration $\mathcal G$ est ou bien nul, ou bien un fibr\'e vectoriel semi-stable de pente $t$.
\end{prop}

Soit $E$ un fibr\'e vectoriel non-nul sur $C$. On suppose que son drapeau de Harder-Narasimhan est
\[0=E_0\subsetneq E_1\subsetneq \ldots\subsetneq E_n=E.\]
Pour tout $i\in\{1,\ldots,n\}$, soit $\alpha_i$ la pente du sous-quotient $E_i/E_{i-1}$. On d\'esigne par $\nu_E$ la mesure de probabilit\'e bor\'elienne sur $\mathbb R$ d\'efinie comme 
\[\nu_E(\mathrm{d}t)=-\mathrm{d}\frac{\rang(\mathcal F^tE)}{\rang(E)}=\sum_{i=1}^n\frac{\rang(E_i/E_{i-1})}{\rang(E)}\delta_{\alpha_i}.\]
Avec cette notation, on peut r\'e\'ecrire $\deg_+(E)$ comme
\begin{equation}\label{Equ:degplus}
\deg_+(E)=\rang(E)\int_0^{+\infty}t\,\nu_E(\mathrm{d}t)=\int_0^{\mu_{\max}(E)}\rang(\mathcal F^tE)\,\mathrm{d}t,
\end{equation}
o\`u la derni\`ere \'egalit\'e provient de l'int\'egration par partie et du fait que $\mathcal F^tE=0$ quand $t>\mu_{\max}(E)$. Le th\'eor\`eme \ref{Thm:h0etdegplus} montre alors que
\begin{equation}
\bigg|h^0(E)-\int_0^{\mu_{\max}(E)}\rang(\mathcal F^tE)\,\mathrm{d}t\bigg|\leqslant\rang(E)\max(g-1,1).
\end{equation} 

Dans le cas o\`u le corps de base $k$ est de caract\'eristique z\'ero, d'apr\`es un r\'esultat de Narasimhan et Seshadri \cite{Nara_Se65}, le produit tensoriel de deux fibr\'es vectoriels semi-stables sur $C$ est encore semi-stable. On renvoie les lecteurs dans \cite[\S1.1]{Bost_Chen} pour un survol succinct de diff\'erentes approches autour de la semi-stabilit\'e tensorielle dans la litt\'erature. 
Ce r\'esultat implique que la $\mathbb R$-filtration de Harder-Narasimhan du produit tensoriel de deux fibr\'es vectoriels s'identifie \`a la filtration produit.

\begin{prop}\label{Pro:produitdefiltrationHN}
On suppose que le corps $k$ est de caract\'eristique z\'ero. Soient $E$ et $F$ deux fibr\'es vectoriels non-nuls sur $C$. Si $t$ est un nombre r\'eel, alors on a 
\begin{equation}\label{Equ:HNfiltratio}\mathcal F^{t}(E\otimes F)=\sum_{\begin{subarray}{c}
(a,b)\in\mathbb R^2\\
a+b=t
\end{subarray}}\mathcal F^a(E)\otimes\mathcal F^b(F),
\end{equation}
o\`u $\mathcal F$ d\'esigne la $\mathbb R$-filtration de Harder-Narasimhan.
\end{prop}
\begin{proof} On d\'esigne par $\mathcal G$ la $\mathbb R$-filtration de $E\otimes F$ telle que $\mathcal G^t(E\otimes F)$ soit d\'efini comme le membre \`a droite de la formule \eqref{Equ:HNfiltratio}. Notre but est de d\'emontrer que la filtration $\mathcal G$ s'identifie \`a la $\mathbb R$-filtration de Harder-Narasimhan de $E\otimes F$. Pour tout $t\in\mathbb R$, le sous-quotient d'indice $t$ de la filtration $G$ s'\'ecrit sous la forme
\[\sq_{\mathcal G}^t(E\otimes F)=\bigoplus_{a+b=t}\sq^a(E)\otimes\sq^b(F).\]
Le fibr\'e vectoriel $\sq^a(E)$ (resp. $\sq^b(F)$) est ou bien nul, ou bien semi-stable de pente $a$ (resp. $b$). D'apr\`es le r\'esultat de Narasimhan et Seshadri, le produit tensoriel $\sq^a(E)\otimes\sq^b(F)$ est un fibr\'e vectoriel nul ou semi-stable de pente $a+b$. Cela montre que le sous-quotient $\sq_{\mathcal G}^t(E\otimes F)$ est nul ou semi-stable de pente $t$. D'apr\`es la proposition \ref{Pro:criteredeHN}, on obtient que la filtration $\mathcal G$ est la $\mathbb R$-filtration de Harder-Narasimhan de $E$.
\end{proof}

\begin{coro}\label{Cor:filtrationdeHNenalg} On suppose que le corps $k$ est de caract\'eristique z\'ero. 
Soit $E_\sbullet=\bigoplus_{n\geqslant 0}E_n$ une $\mathcal O_C$-alg\`ebre gradu\'ee. On suppose que chaque composante homog\`ene $E_n$ est un fibr\'e vectoriel sur $C$. Alors la $\mathbb R$-filtration de Harder-Narasimhan est compatible \`a la structure de $\mathcal O_C$-alg\`ebre de $E_\sbullet$. Autrement dit, pour tout couple d'entiers $(n,m)\in\mathbb N^2$ et tout $(a,b)\in\mathbb R^2$, on a
\begin{equation}
(\mathcal F^aE_n)\cdot(\mathcal F^bE_m)\subset\mathcal F^{a+b}E_{n+m}.
\end{equation}
\end{coro}
\begin{proof}
Par d\'efinition $(\mathcal F^aE_n)\cdot(\mathcal F^bE_m)$ est l'image canonique de $(\mathcal F^aE_n)\otimes(\mathcal F^bE_m)$ par l'homomorphisme $\varphi_{n,m}:E_n\otimes E_m\rightarrow E_{n+m}$ de la structure de $\mathcal O_C$-alg\`ebre de $E_\sbullet$. La proposition pr\'ec\'edente montre que $(\mathcal F^aE_n)\otimes(\mathcal F^bE_m)$ est contenu dans $\mathcal F^{a+b}(E_n\otimes E_m)$. En outre, d'apr\`es \cite[proposition 2.2.4]{Chen10b}, tout homomorphisme de fibr\'es vectoriels sur $C$ pr\'eserve les $\mathbb R$-filtrations de Harder-Narasimhan. En particulier, on a \[\varphi_{n,m}(\mathcal F^{a+b}(E_n\otimes E_m))\subset\mathcal F^{a+b}(E_{n+m}),\]
d'o\`u le r\'esultat.
\end{proof}

\section{Pente maximal asymptotique}

Soient $C$ une courbe projective r\'eguli\`ere d\'efinie sur un corps $k$, et $\pi:X\rightarrow C$ un morphisme projectif et plat d'un sch\'ema int\`egre $X$ vers $C$. On s'int\'eresse \`a des invariants birationnels de faisceaux inversibles sur $X$. Rappelons que le \emph{volume} d'un faisceau inversible $L$ sur $X$ est d\'efini comme 
\begin{equation}
\vol(L):=\limsup_{n\rightarrow+\infty}\frac{\rang_kH^0(X,L^{\otimes n})}{n^{\dim(X)}/\dim(X)!}.
\end{equation}
On dit que le faisceau inversible $L$ est \emph{gros} si son volume est strictement positif. La fonction volume est un invariant birationnel (cf. \cite[proposition 2.2.43]{LazarsfeldI})~: si $p:X'\rightarrow X$ est un morphisme projectif birationnel d'un sch\'ema int\`egre $X'$ vers $X$, alors on a $\vol(p^*(L))=\vol(L)$. En outre, le faisceau inversible $L$ est gros\label{Page:criteregors} si et seulement si une puissance tensorielle de $L$ peut \^etre d\'ecompos\'ee en le produit tensoriel d'un faisceau inversible ample et un faisceau inversible effectif (i.e. qui poss\`ede au moins une section globale non-nulle). On renvoie les lecteurs dans \cite[corollaire 2.2.7]{LazarsfeldI} pour une d\'emonstration. Ce crit\`ere montre en particulier que les faisceaux inversibles gros forment un c\^one ouvert dans le groupe de Picard de $X$~: si $L$ est un faisceau inversible gros et si $M$ est un faisceau inversible quelconque sur $X$, alors pour tout entier $n$ assez positif, le produit tensoriel $L^{\otimes n}\otimes M$ est un faisceau inversible gros.

Soit $K$ le corps des fonctions rationnelles sur la courbe $C$. On d\'esigne par $\eta:\Spec K\rightarrow C$ le point g\'en\'erique de $C$. Soit $L$ un faisceau inversible sur $X$ tel que $L_\eta$ soit gros. Dans \cite[th\'eor\`eme 4.3.6]{Chen10b}, il est d\'emontr\'e que la suite $\big(\mu_{\max}(\pi_*(L^{\otimes n}))/n\big)_{n\geqslant 1}$ converge dans $\mathbb R$. On d\'esigne par $\mu_{\max}^{\pi}(L)$ la limite de cette suite, appel\'ee la \emph{pente maximale asymptotique} de $L$ relativement \`a $\pi$. Par d\'efinition, on a $\mu_{\max}^{\pi}(L^{\otimes n})=n\mu_{\max}^{\pi}(L)$ pour tout entier $n\geqslant 1$. En outre, si $M$ est un faisceau inversible sur $C$, alors on a (cf. \cite[proposition 4.3.8]{Chen10b})
\begin{equation}\label{Equ:mupimaxtordu}\mu_{\max}^{\pi}(L\otimes\pi^*(M))=\mu_{\max}^{\pi}(L)+\deg(M).\end{equation}

\begin{lemm}
Soit $E$ un fibr\'e vectoriel non-nul sur la courbe $C$. Si $\mu_{\max}(E)>|g-1|$, o\`u $g$ est le genre de $C$, alors on a $h^0(E)>0$.
\end{lemm}
\begin{proof}
Soit $E_{\des}$ le sous-fibr\'e d\'estabilisant du fibr\'e vectoriel $E$. On a $\mu_{\min}(E_{\des})=\mu(E_{\des})=\mu_{\max}(E)>0$. D'apr\`es le lemme \ref{Lem:h0} (c), on obtient
\[h^0(E)\geqslant h^0(E_{\des})\geqslant \rang(E_{\des})\mu_{\max}(E)-\rang(E_{\des})|g-1|,\]
d'o\`u le r\'esultat.
\end{proof}

\begin{rema}\label{Rem:grosseur}
Soit $L$ un faisceau inversible sur $X$ qui est g\'en\'eriquement gros. Le lemme pr\'ec\'edent montre que, si $\mathrm{\mu}_{\max}^{\pi}(L)>0$, alors pour tout entier $n$ suffisamment positif, le faisceau inversible $L^{\otimes n}$ est effectif. En effet, d'apr\`es la d\'efinition de $\mu_{\max}^{\pi}(L)$, la pente maximale de $\mu_{\max}(\pi_*(L^{\otimes n}))$ cro\^it lin\'eairement par rapport \`a $n$ lorsque $n$ tend vers l'infini. Elle d\'epasse  $|g-1|$ lorsque $n$ est assez positif.
\end{rema}

\begin{prop}\label{Pro:criteredegros}
Soit $L$ un faisceau inversible sur $X$. Alors $L$ est gros si et seulement si $L_\eta$ est gros et $\mu_{\max}^{\pi}(L)>0$.
\end{prop}
\begin{proof}
``$\Longrightarrow$'': Soit $L$ un faisceau inversible gros sur $X$. D'apr\`es le crit\`ere de grosseur que l'on a mentionn\'e plus haut dans la page \pageref{Page:criteregors}, il existe un entier $n\geqslant 1$, un faisceau inversible ample $A$ et un faisceau inversible effectif $L'$ sur $X$ tels que $L^{\otimes n}\cong A\otimes L'$. La restriction de cette formule \`a la fibre g\'en\'erique de $X$ donne une d\'ecomposition de $L_\eta^{\otimes n}$ en produit tensoriel d'un faisceau inversible ample et un faisceau inversible effectif. Cela montre que $L_\eta$ est gros. 

Comme $L$ est un faisceau inversible gros, pour tout entier $n$ assez positif, $L^{\otimes n}$ poss\`ede au moins une section globale non-nulle (cf. la remarque \ref{Rem:grosseur}). Par cons\'equent, on a $\mu_{\max}(\pi_*(L^{\otimes n}))\geqslant 0$ pour tout entier $n$ assez positif. On en d\'eduit $\mu_{\max}^{\pi}(L)\geqslant 0$. Soit $M$ un faisceau inversible sur $C$ tel que $\deg(M)>0$. Comme $L$ est gros, il existe un entier $n\geqslant 1$ tel que $L^{\otimes n}\otimes\pi^*( M^\vee)$ soit gros. On a alors 
$\mu_{\max}^{\pi}(L^{\otimes n}\otimes \pi^*(M^\vee))\geqslant 0$. D'apr\`es \eqref{Equ:mupimaxtordu}, cela implique que
\[n\mu_{\max}^{\pi}(L)\geqslant\deg(M)>0,\]
d'o\`u $\mu_{\max}^{\pi}(L)>0$.

``$\Longleftarrow$'': On suppose que $L$ est un faisceau inversible sur $X$ qui est g\'en\'eriquement gros et tel que $\mu_{\max}^{\pi}(L)>0$. On fixe un faisceau inversible ample $A$ sur $X$. Comme $L_\eta$ est suppos\'e \^etre gros, il existe un entier $d\geqslant 1$ tel que $L_\eta^{\otimes d}\otimes A_\eta^\vee$ poss\`ede une section globale non-nulle $s$. La section $s$ se rel\`eve en une section rationnelle de $L^{\otimes d}\otimes A^\vee$ dont le diviseur est effectif \`a un diviseur vertical pr\`es. Il existe alors un faisceau inversible ample $M$ sur $C$ tel que $s$ se prolonge en une section globale non-nulle de $L^{\otimes d}\otimes A^\vee\otimes\pi^*(M)$. En outre, comme $\mu_{\max}^{\pi}(L)>0$, il existe un entier $n\geqslant 1$ tel que
\[\mu_{\max}^{\pi}(L^{\otimes n}\otimes\pi^*(M^\vee))=n\mu_{\max}^{\pi}(L)-\deg(M)>0.\]
D'apr\`es le lemme pr\'ec\'edent (voir aussi la remarque qui le suit), il existe alors un entier $m\geqslant 1$ tel que $L^{\otimes nm}\otimes\pi^*(M^\vee)^{\otimes m}$ soit effectif. On en d\'eduit que le faisceau inversible 
\[L^{md+nm}\otimes (A^{\otimes m})^\vee\cong (
L^{\otimes d}\otimes A^\vee\otimes\pi^*(M))^{\otimes m}\otimes(L^{\otimes nm}\otimes\pi^*(M^\vee)^{\otimes m})\]
admet une section globale non-nulle. Cela montre que $L$ est un faisceau inversible gros. 
\end{proof}

\begin{coro}\label{Cor:mumaxbir}
La pente maximale asymptotique est un invariant birationnel pour les faisceaux inversibles g\'en\'eriquement gros sur $X$.
\end{coro}
\begin{proof}
Soit $M$ un faisceau inversible ample sur $C$. On affirme que, pour tout faisceau inversible $L$ sur $X$ qui est g\'en\'eriquement gros, la valeur $\mu_{\max}^{\pi}(L)$ est \'egale \`a
\[\sup\Big\{\frac{n\deg(M)}{m}\,\Big|\,(n,m)\in\mathbb N_{\geqslant 1}^2,\,L^{\otimes m}\otimes\pi^*(M^\vee)^{\otimes n}\text{ est gros}\Big\}.\]
En effet, d'apr\`es la proposition pr\'ec\'edente, $L^{\otimes m}\otimes \pi^*(M^\vee)^{\otimes n}$ est gros si et seulement si
\[\mu_{\max}^{\pi}(L^{\otimes m}\otimes \pi^*(M^\vee)^{\otimes n})=m\mu_{\max}^{\pi}(L)-n\deg(M)>0,\]
ou de fa\c{c}on \'equivalente, $\mu_{\max}^{\pi}(L)>(n/m)\deg(M)$. Comme la fonction volume est un invariant birationnel, on obtient que, pour tout morphisme projectif et birationnel $p:X'\rightarrow X$, le faisceau inversible $L^{\otimes m}\otimes\pi^*(M^\vee)^{\otimes n}$ est gros si et seulement si \[p^*(L^{\otimes m}\otimes\pi^*(M^{\vee})^{\otimes n})\cong p^*(L^{\otimes m})\otimes (\pi p)^*(M^\vee)^{\otimes n}\] l'est. D'o\`u $\mu_{\max}^{\pi}(L)=\mu_{\max}^{\pi}(p^*L)$.
\end{proof}

Dans la suite, on g\'en\'eralise la construction de pente maximale asymptotique aux syst\`emes gradu\'es en fibr\'es vectoriels. Soit $L$ un faisceau inversible sur $X$. On entend par \emph{syst\`eme gradu\'e en fibr\'es vectoriels} de $L$ toute sous-$\mathcal O_C$-alg\`ebre gradu\'ee de $\bigoplus_{n\geqslant 0}\pi_*(L^{\otimes n})$. Si $E_\sbullet$ est un syst\`eme gradu\'e en fibr\'es vectoriels de $L$, alors sa fibre g\'en\'erique $E_{\sbullet,\eta}:=\bigoplus_{n\geqslant 0}E_{n,\eta}$ est un syst\`eme lin\'eaire gradu\'e de $L_\eta$. On dit que $E_{\sbullet,\eta}$ \emph{contient un diviseur ample} si les conditions suivantes sont satisfaites (cf. la condition (C) dans \cite[d\'efinition 2.9]{Lazarsfeld_Mustata08})~:
\begin{enumerate}[(a)]
\item l'espace vectoriel $E_{n,\eta}$ sur $K$ est non-nul pour tout entier $n$ suffisamment positif,
\item il existe un faisceau inversible ample $A_\eta$ sur $X_\eta$, un entier $p\geqslant 1$ et une section globale non-nulle $s$ de $L^{\otimes p}_\eta\otimes A_\eta^\vee$ tels que, pour tout $n\in\mathbb N$, l'image de l'homomorphisme
\[\xymatrix{\relax H^0(X_\eta,A_\eta^{\otimes n})\ar[rr]^-{\cdot s^{n}}&&H^0(X_\eta,L_\eta^{\otimes np})}\]
soit contenue dans $E_{np,\eta}$.
\end{enumerate}
Cette condition revient \`a la grosseur de $L_\eta$ lorsque $E_\sbullet=\bigoplus_{n\geqslant 0}\pi_*(L^{\otimes n})$ est le syst\`eme gradu\'e total. Si la fibre g\'en\'erique de $E_\sbullet$ contient un diviseur ample, alors la suite de pentes maximales normalis\'ee $(\mu_{\max}(E_n)/n)_{n\geqslant 1}$ converge dans $\mathbb R$ (cf. \cite[th\'eor\`eme 4.3.1]{Chen10b}). On d\'esigne par $\mu_{\max}^{\mathrm{asy}}(E_\sbullet)$ sa limite. Pour tout entier $n\geqslant 0$, le fibr\'e vectoriel $E_n$ est un sous-fibr\'e vectoriel de $\pi_*(L^{\otimes n})$. On obtient donc $\mu_{\max}(E_n)\leqslant\mu_{\max}(\pi_*(L^{\otimes n}))$. Cela implique que $\mu_{\max}^{\mathrm{asy}}(E_\sbullet)\leqslant\mu_{\max}^{\pi}(L)$.
\begin{rema}\label{Rem:systemegradue}
Comme la fibre g\'en\'erique de $E_\sbullet$ est un anneau int\`egre, on obtient du corollaire \ref{Cor:filtrationdeHNenalg} que, dans le cas o\`u le corps $k$ est de caract\'eristique $0$, si $E_n$ et $E_m$ sont non-nuls, alors on a\footnote{En effet, l'homomorphisme naturel de $E_{n,\des}\otimes E_{m,\des}$ vers $E_{n+m}$ est non-nul. En outre, le corollaire \ref{Cor:filtrationdeHNenalg} montre que $E_{n,\des}\otimes E_{m,\des}$ est semi-stable de pente $\mu_{\max}(E_n)+\mu_{\max}(E_m)$. On obtient donc l'in\'egalit\'e souhait\'ee.} 
\[\mu_{\max}(E_{n+m})\geqslant\mu_{\max}(E_n)+\mu_{\max}(E_m).\]
Cela montre que $\mu_{\max}^{\mathrm{asy}}(E_\sbullet)\geqslant \mu_{\max}(E_n)/n$ d\`es que $E_n$ est non-nul ($n\geqslant 1$).
\end{rema}

Soient $L$ un faisceau inversible sur $X$ et $E_\sbullet$ un syst\`eme gradu\'e en fibr\'es vectoriels de $L$. On d\'efinit le \emph{volume} de $E_\sbullet$ comme
\[\vol(E_\sbullet):=\limsup_{n\rightarrow+\infty}\frac{h^0(E_n)}{n^{\dim(X)}/\dim(X)!}.\]
Lorsque $E_\sbullet$ est le syst\`eme gradu\'e total $\bigoplus_{n\geqslant 0}\pi^*(L^{\otimes n})$, son volume s'identifie au volume de $L$.

Soient $X$ un sch\'ema projectif et int\`egre sur $\Spec k$ et $L$ un faisceau inversible gros sur $X$. Soit $V_\sbullet=\bigoplus_{n\geqslant 0}V_n$ un syst\`eme lin\'eaire gradu\'e de $L$ (i.e. une sous-alg\`ebre gradu\'ee de $\bigoplus_{n\geqslant 0}H^0(X,L^{\otimes n})$). On suppose que $V_\sbullet$ contient un diviseur ample, c'est-\`a-dire que $V_n\neq 0$ pour $n$ assez positif, et qu'il existe un faisceau inversible ample $A$ sur $X$, un entier $p\geqslant 1$ et une section globale non-nulle $s$ de $L^{\otimes p}\otimes A^\vee$ tels que, pour tout $n\in\mathbb N$, on a 
\[\Image(\xymatrix{\relax H^0(X,A^{\otimes n})\ar[r]^-{\cdot s^{n}}&H^0(X,L^{\otimes np}))}\subset V_{np}.\]
Rappelons que le \emph{volume} de $V_\sbullet$ est d\'efini comme 
\[\vol(V_\sbullet)=\limsup_{n\rightarrow+\infty}\frac{\rang_k(V_n)}{n^{\dim(X)}/(\dim X)!}.\]
Pour tout entier $n\geqslant 0$, on d\'esigne par $V_{n,K}$ le sous-$K$-espace vectoriel de $H^0(X_\eta,L_\eta^{\otimes n})$ engendr\'e par l'image canonique $V_n$. Il s'av\`ere que $V_{\sbullet,K}:=\bigoplus_{n\geqslant 0}V_{n,K}$ est un syst\`eme lin\'eaire gradu\'e de $L_\eta$ qui contient un diviseur ample.

Dans la suite, on construit un syst\`eme gradu\'e en fibr\'es vectoriels $E_\sbullet$ tel que $E_{\sbullet,\eta}$ coincide \`a $V_{\sbullet,K}$ et que $\vol(E_\sbullet)=\vol(V_\sbullet)$.

\begin{theo}\label{Thm:passageauxfibres} 
Soit $V_\sbullet$ un syst\`eme lin\'eaire gradu\'e de $L$ qui contient un diviseur ample. Pour tout entier $n\geqslant 0$, soit $E_n$ le sous-$\mathcal O_C$-module de $\pi_*(L^{\otimes n})$ engendr\'e\footnote{C'est-\`a-dire que $E_n$ est l'image de l'homomorphisme $\varphi^*(V_n)\rightarrow\pi_*(L^{\otimes n})$ induit par l'inclusion $V_n\rightarrow H^0(X,L^{\otimes n})=\varphi_*(\pi_*(L^{\otimes n}))$ via l'adjonction entre les foncteurs $\varphi_*$ et $\varphi^*$, o\`u $\varphi:C\rightarrow\Spec k$ d\'esigne le morphisme structurel.} par $V_n$.  Alors $E_\sbullet=\bigoplus_{n\geqslant 0}E_n$ est un syst\`eme gradu\'e en fibr\'es vectoriels de $L$, dont la fibre g\'en\'erique contient un diviseur ample. De plus, on a
$\mathrm{vol}(E_\sbullet)=\mathrm{vol}(V_\sbullet)$.
\end{theo}
\begin{proof} Sans perte de g\'en\'eralit\'e, on peut supposer $X$ normal. En effet, par passage \`a la normalisation $\nu:\widetilde X\rightarrow X$, on peut consid\'erer $V_\sbullet$ comme un syst\`eme lin\'eaire gradu\'e de $\nu^*(L)$, qui contient un diviseur ample.

On d\'esigne par $\varphi:C\rightarrow\Spec k$ le morphisme structurel.
Par d\'efinition $E_\sbullet$ est l'image de l'homomorphisme de $\mathcal O_C$-alg\`ebre gradu\'ee $\bigoplus_{n\geqslant 0}\varphi^*(V_n)\rightarrow\bigoplus_{n\geqslant 0}\pi_*(L^{\otimes n})$. Donc il est une sous-$\mathcal O_C$-alg\`ebre gradu\'ee de $\bigoplus_{n\geqslant 0}\pi_*(L^{\otimes n})$, i.e., un syst\`eme gradu\'e en fibr\'es vectoriels de $L$. En outre, il existe un entier $p\geqslant 1$ et  un faisceau inversible ample $A$ sur $X$ tels que le faisceau inversible $L^{\otimes p}\otimes A^\vee$ poss\`ede une section globale non-nulle $s$ v\'erifiant
\begin{equation}\label{Equ:contientundivamp}\Image(\xymatrix{\relax H^0(X,A^{\otimes n})\ar[r]^-{\cdot s^{n}}&H^0(X,L^{\otimes np})})\subset V_{np}\end{equation}
pour tout entier $n\geqslant 1$. Comme $A$ est ample, pour tout entier $m$ suffisamment positif, le $K$-espace vectoriel $H^0(X_\eta,A_\eta^{\otimes m})$ (o\`u $\eta$ est le point g\'en\'erique de $C$) est engendr\'e\footnote{Cela provient d'un analogue dans le cadre de corps de fonction du corollaire 4.8 de \cite{Zhang95}. On peut suivre la strat\'egie de \emph{loc. cit.}. La d\'emonstration est plus simple car les places archim\'edienne ne se manifestent pas dans le probl\`eme.} par $H^0(X,A^{\otimes m})$. Quitte \`a remplacer $A$ par l'une de ses puissance tensorielle, on peut supposer que cette propri\'et\'e est v\'erifi\'ee pour tout entier $m\geqslant 1$. On d\'eduit alors de la relation \eqref{Equ:contientundivamp} que
\[\Image(\xymatrix{\relax H^0(X_\eta,A_\eta^{\otimes n})\ar[r]^-{\cdot s_\eta^{n}}&H^0(X_\eta,L_\eta^{\otimes np})})\subset E_{np,K}.\]
Cela montre que le syst\`eme lin\'eaire $E_{\sbullet,\eta}$ contient un diviseur ample.

Comme le syst\`eme lin\'eaire gradu\'e $V_\sbullet$ contient un diviseur ample, on obtient que, pour tout entier $n$ assez positif, le morphisme rationnel de $X$ vers $\mathbb P(V_n)$ d\'efini par le syst\`eme lin\'eaire est birationnel. Si $p\geqslant 1$ est un entier, on d\'esigne par $u_p:X_p\rightarrow X$ l'\'eclatement de $X$ le long du lieu de base de $V_p$, d\'efini comme
\[X_p=\mathrm{Proj}\Big(\bigoplus_{n\geqslant 0}(\varphi\pi)^*(\mathrm{Sym}^n(V_p))\longrightarrow L^{\otimes np}\Big).\]
Soient en outre $j_p:X_p\rightarrow\mathbb P(V_p)$ le morphisme canonique et $L_p$ le tire en arri\`ere du faisceau universel $\mathcal O_{V_p}(1)$ \`a $X_p$. Il existe alors $N\in\mathbb N$ tel que $j_p$ d\'efinisse un morphisme birationnel entre $X_p$ et son image dans $\mathbb P(V_p)$ d\`es que $p>N$. Soit $p$ un tel entier. Consid\'erons les homomorphisme de $k$-espaces vectoriels comme ci-dessous
\[\xymatrix{\relax\mathrm{Sym}^n(V_p)\ar[r]&H^0(j_p(X_p),\mathcal O_{V_p}(n))\ar[r]&H^0(X_p,L_p^{\otimes n})\ar[r]&H^0(X_p,u_p^*(L^{\otimes np}))}.\] 
Le premier morphisme est surjectif pour $n$ assez positif car on peut identifier $\mathrm{Sym}^n(V_p)$ \`a $H^0(\mathbb P(V_p),\mathcal O_{V_p}(n))$.   Le deuxi\`eme homomorphisme est injectif car $j_p:X_p\rightarrow j_p(X_p)$ est un morphisme birationnel et $L_p^{\otimes n}\cong j_p^*\mathcal O_{V_p}(n)$. Le dernier homomorphisme est d\'efini comme la multiplication par la $k^{\text{i\`eme}}$ puissance  de la section qui d\'etermine le diviseur exceptionnel de l'\'eclatement $u_p:X_p\rightarrow X$, donc est aussi injectif. Si on identifie $H^0(X_p,u_p^*(L^{\otimes np}))$ \`a $H^0(X,L^{\otimes np})$ (on peut faire \c{c}a car le sch\'ema $X$ est suppos\'e \^etre normal, cf. \cite[corollaire 4.3.12]{EGAIII_1}), l'image de l'homomorphisme compos\'e s'identifie \`a $\Image(\mathrm{Sym}^pV_n\rightarrow V_{np})$. En outre, comme le morphisme $j_p:X_p\rightarrow j_p(X_p)$ est birationnel, on a 
$\mathrm{vol}(L_p)=\mathrm{vol}(\mathcal O_{V_p}(n)|_{j_p(X_p)})$.  Cela montre que le volume du faisceau inversible $L_p$ est \'egale \`a celui du syst\`eme lin\'eaire gradu\'e
\[V^{[p]}_\sbullet:=\bigoplus_{n\geqslant 0}\Image(\mathrm{Sym}^n(V_p)\longrightarrow V_{np}).\]
Soit $E^{[p]}_{\sbullet}:=\bigoplus_{n\geqslant 0}\Image(\mathrm{Sym}^n(E_p)\rightarrow E_{np})$. C'est un syst\`eme gradu\'e en fibr\'es vectoriels de $L^{\otimes p}$. Pour tout entier $n$ assez positif, on a $E_n^{[p]}\subset(\pi u_p)_*(L_p^{\otimes n})$. On obtient alors 
\[\mathrm{vol}(V_\sbullet)\geqslant\frac{\mathrm{vol}(V^{[p]}_{\sbullet})}{p^{\dim(X)}}=\frac{\mathrm{vol}(L_p)}{p^{\dim(X)}}\geqslant\frac{\mathrm{vol}(E^{[p]}_{\sbullet})}{p^{\dim(X)}}.\]
D'apr\`es le th\'eor\`eme d'approximation de Fujita pour les syst\`emes lin\'eaires gradu\'es en fibr\'es ad\'eliques (qui sont plus g\'en\'eraux que les fibr\'es vectoriels, cf. \cite[th\'eor\`eme 2.9]{Boucksom_Chen}), on a 
\[\sup_{p\geqslant 1}\frac{\mathrm{vol}(E_\sbullet^{[p]})}{p^{\dim(X)}}=\mathrm{vol}(E_\sbullet).\]
On obtient donc $\vol(V_\sbullet)\geqslant\vol(E_\sbullet)$. Enfin, comme $V_n\subset H^0(C,E_n)$ pour tout $n\in\mathbb N$, on obtient $\vol(V_\sbullet)\leqslant\vol(E_\sbullet)$. La d\'emonstration est donc achev\'ee.
\end{proof}

\begin{defi}\label{Def:imagedirecte}
Soient $X$ un sch\'ema int\`egre projectif d\'efini sur $k$ et $L$ un faisceau inversible gros sur $X$. Soit $V_\sbullet$ un syst\`eme lin\'eaire gradu\'e de $L$ qui contient un diviseur ample. Pour tout entier $n\geqslant 0$, soit $\pi_*(V_n)$ le sous-$\mathcal O_C$-module de $\pi_*(L^{\otimes n})$ engendr\'e par $V_n$. On d\'esigne par $\pi_*(V_\sbullet)$ le syst\`eme gradu\'e en fibr\'es vectoriels $\bigoplus_{n\geqslant 0}\pi_*(V_n)$ et par $\mu_{\max}^\pi(V_\sbullet)$ la quantit\'e $\mu_{\max}^{\mathrm{asy}}(\pi_*(V_\sbullet))$, appel\'ee la \emph{pente maximale asymptotique} de $V_\sbullet$ relativement \`a $\pi$.
\end{defi}

\begin{rema}\label{Rem:comparaisonmu}
Soient $L$ et $M$ deux faisceau inversibles gros sur $X$, et $V_\sbullet$ et $W_\sbullet$ des syst\`emes lin\'eaires gradu\'es de $L$ et $M$ respectivement. On suppose que $V_\sbullet$ et $W_\sbullet$ contiennent des diviseurs amples. S'il existe une section non-nulle $s$ de $L^\vee\otimes M$ telle que 
\[\forall\,n\geqslant 1,\quad
\Image(\xymatrix{\relax V_n\ar[r]^-{s^n\cdot}&H^0(X,M^{\otimes n})})\subset W_n,\]
on dit que $V_\sbullet$ est \emph{contenu} dans $W_\sbullet$ (via la section $s$). Il s'av\`ere que la multiplication par $s^n$ d\'efinit aussi un homomorphisme injectif de $\pi_*(V_n)$ vers $\pi_*(W_n)$. On obtient donc $\mu_{\max}(V_n)\leqslant\mu_{\max}(W_n)$, qui implique la relation $\mu_{\max}^{\pi}(V_\sbullet)\leqslant\mu_{\max}^{\pi}(W_\sbullet)$. Similairement, on   a $\vol(V_\sbullet)\leqslant \vol(W_\sbullet)$.
\end{rema}

\section{Tour de fibrations sur courbes}\label{Sec:tourdefibration}

Soient $k$ un corps  et $X$ un sch\'ema projectif et int\`egre de dimension $d+1$ sur $\Spec k$, o\`u $d$ est un entier, $d\geqslant 0$. Par \emph{tour de fibrations sur courbes} de $X$, on entend toute donn\'ee
$(p_i:X_i\rightarrow C_i)_{i=0}^{d}$ o\`u
\begin{enumerate}[(a)]
\item lorsque $d=0$,  $C_0$  et $X_0$ sont tous les deux la normalisation du sch\'ema $X$ et $p_0:X_0\rightarrow C_0$ est le morphisme d'identit\'e, 
\item lorsque $d\geqslant 1$, $C_0$ est une courbe projective r\'eguli\`ere sur $\Spec k$, $X_0=X$ et $p_0$ est un $k$-morphisme projectif et plat\footnote{Ici la platitude est \'equivalente \`a la surjectivit\'e du morphisme, cf. \cite[proposition 4.3.9]{LiuQing}.} de $X$ vers $C_0$,
\item de fa\c{c}on r\'ecursive, pour tout $i\in\{1,\ldots,d-1\}$, $C_{i}$ est une courbe projective r\'eguli\`ere d\'efinie sur le corps $R(C_{i-1})$ des fonctions rationnelles sur $C_{i-1}$, $X_{i}$ est la fibre g\'en\'erique de $p_{i-1}$ et $p_{i}:X_{i}\rightarrow C_{i}$ est un morphisme projectif et plat de $R(C_{i-1})$-sch\'emas,
\item $C_{d}$ est la normalisation de la fibre g\'en\'erique de  $p_{d-1}$ et $p_d:C_{d}\rightarrow C_d$ est le morphisme d'identit\'e.
\end{enumerate}

Si $\Theta=(p_i:X_i\rightarrow C_i)_{i=0}^{d-1}$ est un tour de fibrations sur courbes du sch\'ema $X$, on d\'esigne par $g(\Theta)$ le vecteur $(g(C_0),\ldots,g(C_d))\in\mathbb N^{d+1}$, o\`u $g(C_i)$ est le genre de la courbe $C_i$. Le vecteur $g(\Theta)$ est appel\'e le \emph{genre} de $\Theta$. 

\begin{rema}
Soient $X$ un sch\'ema projectif et int\`egre de dimension $d+1$ sur $\Spec k$, o\`u $d\geqslant 2$. Si $(p_i:X_i\rightarrow C_i)_{i=0}^{d}$ est un tour de fibrations sur courbes du $k$-sch\'ema $X$, alors $(p_j:X_j\rightarrow C_j)_{j=i}^d$ est un tour de fibrations sur courbes du $R(C_{i-1})$-sch\'ema $X_i$.

Si le sch\'ema $X$ est normal, alors $C_d$ s'identifie \`a la fibre g\'en\'erique de $p_{d-1}$ (cela provient de la pr\'eservation de la cl\^oture int\'egrale par la localisation).
\end{rema}

On entend par \emph{modification birationnelle} d'un $k$-sch\'ema projectif et int\`egre $X$ tout morphisme $f:X'\rightarrow X$ d'un sch\'ema projectif et int\`egre $X'$ vers $X$ qui est birationnel (autrement dit, $f$ induit un isomorphisme entre les corps des fonctions rationnelles de $X$ et de $X'$). La proposition suivante montre que l'existence d'un tour de fibrations sur courbes de genre fix\'e est une propri\'et\'e invariante par toute modification birationnelle.

\begin{prop}
Soit $X$ un sch\'ema projectif et int\`egre de dimension $d+1$ sur $\Spec k$. On suppose que le sch\'ema $X$ admet un tour de fibrations sur courbes de genre $(g_0,\ldots,g_d)$. Alors, pour toute modification birationnelle $f:X'\rightarrow X$, le sch\'ema $X'$ admet aussi un tour de fibrations sur courbes de m\^eme genre.
\end{prop}
\begin{proof}
On suppose que $C_0$ est une courbe projective r\'eguli\`ere de genre $g_0$ sur $\Spec k$ et que $p_0:X\rightarrow C_0$ est un $k$-morphisme projectif et plat. Alors le morphisme compos\'e $fp_0$ est un morphisme projectif et plat de $X'$ vers $C_0$. De plus, le morphisme canonique de la fibre g\'en\'erique de $fp_0$ vers celle de $f$ est un $R(C_0)$-morphisme projectif et birationnel, o\`u $R(C_0)$ d\'esigne le corps des fonctions rationnelles sur $C_0$. Par r\'ecurrence sur la dimension de $X$, on obtient le r\'esultat.
\end{proof}

Soient $X$ un sch\'ema projectif et int\`egre sur $\Spec k$ et $\Theta$ un tour de fibrations sur courbes du sch\'ema $X$. La proposition pr\'ec\'edente non seulement montre que toute modification birationnelle admet un tour de fibrations sur courbes de m\^eme genre que celui de $\Theta$, sa d\'emonstration construit effectivement un tel tour de fibrations sur courbes pour toute modification birationnelle $f:X'\rightarrow X$, que l'on notera comme $f^*(\Theta)$. Si $\Theta$  est de la forme $(p_i:X_i\rightarrow C_i)_{i=0}^d$, alors le $i^{\text{\`eme}}$ morphisme dans $f^*(\Theta)$ est obtenu comme le compos\'e de $p_i$ avec une modification birationnelle de $X_i$. En outre, on peut v\'erifier que, si $f_1:X'\rightarrow X$ et $f_2:X''\rightarrow X'$ sont des modifications birationnelles successives, alors on a $(f_1f_2)^*\Theta=f_2^*(f_1^*\Theta)$. 

\begin{rema}
\'Etant donn\'e un sch\'ema projectif et int\`egre $X$ d\'efini sur $k$, le choix d'un tour $\Theta=(p_i:X_i\rightarrow C_i)_{i=0}^d$ de fibrations  sur courbes de $X$ d\'efinit une cha\^ine 
\[k\subset R(C_0)\subset\ldots\subset R(C_d)=R(X)\]
de sous-extension de $R(X)/k$, o\`u $R(X)$ est le corps des fonctions rationnelles sur $X$. Chaque extension cons\'ecutive dans la cha\^ine est transcendante de degr\'e  de transcendance $1$. R\'eciproquement, si on fixe une cha\^ine 
\[k=k_{-1}\subset k_0\subset\ldots \subset k_d=R(X)\]
de sous-extension de $R(X)/k$ de sorte que chaque extension $k_i/k_{i-1}$ est transcendante de degr\'e de transcendance $1$, alors il existe une modification birationnelle $X'$ de $X$ qui poss\`ede un tour de fibrations sur courbes $\Theta=(p_i:X_i'\rightarrow C_i)_{i=0}^d$ tel que $k_i=R(C_i)$ quel que soit $i\in\{0,\ldots,d\}$. En effet, on peut choisir $C_0$ comme la courbe projective r\'eguli\`ere d\'efinie sur $\Spec k$ telle que $R(C_0)=k_0$. L'inclusion de $k_0$ dans $R(X)$ d\'efinit un $k$-morphisme rationnel de $X$ vers $C_0$. Quitte \`a \'eclater le lieu o\`u ce morphisme rationnel n'est pas d\'efini, on obtient une modification birationnelle de $X$ muni d'un $k$-morphisme projectif et plat vers $C_0$. Par un proc\'ed\'e de r\'ecurrence, on peut construire une modification birationnelle de $X$ qui poss\`ede un tour de fibrations sur courbes v\'erifiant les propri\'et\'es comme ce que l'on a d\'ecrit plus haut.
\end{rema}

Soit $X$ un sch\'ema projectif et int\`egre de dimension $d+1$ sur $\Spec k$. On suppose que $X$ admet un tour de fibrations sur courbes $\Theta=(p_i:X_i\rightarrow C_i)_{i=0}^d$. Soient $L$ un faisceau inversible gros sur $X$ et $V_\sbullet$ un syst\`eme lin\'eaire gradu\'e de $L$, qui contient un diviseur ample. Si $d=0$, on d\'esigne par $\vol^{\Theta}(V_\sbullet)$ le volume de $V_\sbullet$ et par  $\mu^{\Theta}_{\max}(V_\sbullet)$ la pente maximale asymptotique de $V_\sbullet$ relativement au morphisme d'identit\'e de la normalisation de $X$, o\`u on consid\`ere $V_\sbullet$ comme un syst\`eme lin\'eaire gradu\'e du tire en arri\`ere du faisceau inversible $L$ sur la normalisation de $X$.

Si $d\geqslant 1$, de fa\c{c}on r\'ecursive, on d\'esigne par $\vol^{\Theta}(V_\sbullet)$ le vecteur \[\big(\vol(V_\sbullet),\vol^{\Theta'}(p_{0*}(V_\sbullet)_{\eta_0})\big),\]
o\`u $\Theta'=(p_i:X_i\rightarrow C_i)_{i=1}^d$ (qui est un tour de fibrations sur courbes de $X_1$), $\eta_0$ est le point g\'en\'erique de la courbe $C_0$ et $p_{0*}(V_\sbullet)_{\eta_0}$ est la fibre g\'en\'erique de $p_{0*}(V_\sbullet)$, qui est un syst\`eme lin\'eaire gradu\'e du tire en arri\`ere de $L$ sur $X_1$ contenant un diviseur ample. De fa\c{c}on similaire, on d\'esigne par $\mu_{\max}^{\Theta}(V_\sbullet)$ le vecteur\footnote{cf. la d\'efinition \ref{Def:imagedirecte} pour la construction de $\mu_{\max}^{p_0}(V_\sbullet)$.}
\[\big(\mu_{\max}^{p_0}(V_\sbullet),\mu_{\max}^{\Theta'}(p_{0*}(V_\sbullet)_{\eta_0})\big).\]

Si $f:X'\rightarrow X$ est une modification birationnelle du sch\'ema $X$, alors on a 
\[\vol^{f^*\Theta}(V_\sbullet)=\vol^{\Theta}(V_\sbullet)\quad\text{et}\quad
\mu_{\max}^{f^*\Theta}(V_\sbullet)=\mu_{\max}^{\Theta}(V_\sbullet),\] o\`u on consid\`ere $V_\sbullet$ comme un syst\`eme lin\'eaire gradu\'e de $f^*(L)$. 


\begin{rema}
Soient $X$ un sch\'ema projectif et int\`egre de dimension $d+1$ sur $\Spec k$ et $\Theta=(p_i:X_i\rightarrow C_i)_{i=0}^d$ un tour de fibration sur courbes de $X$. Soient $L$ est un faisceau inversible gros sur $X$ et $V_\sbullet$ un syst\`eme lin\'eaire gradu\'e de $L$. Soient $(v_0,\ldots,v_d)$ le vecteur $\vol^{\Theta}(V_\sbullet)$ et $(\mu_0,\ldots,\mu_d)$ le vecteur $\mu_{\max}^{\Theta}(V_\sbullet)$. Alors on a
\[\forall\,i\in\{0,\ldots,d\},\quad v_i\leqslant\vol(L|_{X_i})\quad\text{et}\quad\mu_i\leqslant\mu^{p_i}_{\max}(L|_{X_i})\]
En outre, si $V_\sbullet$ est le syst\`eme lin\'eaire gradu\'e total $\bigoplus_{n\geqslant 0}H^0(X,L^{\otimes n})$, alors on a $\mu_i=\mu_{\max}^{p_i}(L|_{X_i})$ pour tout $i\in\{0,\ldots,d\}$. Cependant, en g\'en\'eral l'\'egalit\'e $v_i=\vol(L|_{X_i})$ n'est pas vraie lorsque $i\geqslant 1$. En effet, le syst\`eme gradu\'e en fibr\'es vectoriels $p_{0*}(V_\sbullet)$ ne tient compte que la partie positive du syst\`eme gradu\'e total $\bigoplus_{n\geqslant 0}p_{0*}(L^{\otimes n})$ (qui permet cependant de retrouver le volume de $V_\sbullet$). Les \'egalit\'es $v_i=\vol(L|_{X_i})$ ($i\in\{0,\ldots,d\}$) sont v\'erifi\'ees notamment lorsque $L$ est ample.
\end{rema} 

\section{Estimation explicite la fonction de Hilbert-Samuel}\label{Sec:HS}

Le but de ce paragraphe est de d\'emontrer le th\'eor\`eme \ref{Thm:HMgeometryique}. On fixe un corps commutatif $k$ de caract\'eristique z\'ero. Soit $X$ un sch\'ema projectif et int\`egre sur $\Spec k$. On suppose donn\'e un tour de fibrations sur courbes $\Theta=(p_i:X_i\rightarrow C_i)_{i=0}^d$ de $X$. Pour tout syst\`eme lin\'eaire gradu\'e $V_\sbullet$ d'un faisceau inversible gros $L$ sur $X$, on introduit un invariant birationnel $\varepsilon^{\Theta}(V_\sbullet)$, construit dans la suite. Si $d=0$, alors on d\'efinit
\[\varepsilon^{\Theta}(V_\sbullet):=\max(g(\Theta)-1,1).\]
Lorsque $d\geqslant 1$, on d\'esigne par $W_\sbullet$ la fibre g\'en\'erique de $p_{0*}(V_\sbullet)$, qui est un syst\`eme lin\'eaire gradu\'e de $L|_{X_1}$ contenant un diviseur ample. Soit en outre $\Theta':=(p_i:X_i\rightarrow C_i)_{i=1}^d$, qui est un tour de fibrations sur courbes de $X_1$. Alors l'invariant $\varepsilon^{\Theta}(V_\sbullet)$ est d\'efinie de fa\c{c}on r\'ecursive comme 
\begin{equation}\label{Equ:termeerreur}\varepsilon^{\Theta}(V_\sbullet)=\mu_0\varepsilon^{\Theta'}(W_\sbullet)+\Big(\frac{\vol(W_\sbullet)}{d!}+\varepsilon^{\Theta'}(W_\sbullet)\Big)\max(g_0-1,1),\end{equation}
o\`u $\mu_0$ et $g_0$ sont respectivement les premi\`eres coordonn\'ees des vecteurs $\mu_{\max}^{\Theta}(V_\sbullet)$ et $g(\Theta)$. On voit aussit\^ot de la d\'efinition que, si $V_\sbullet$ est contenu dans un autre syst\`eme lin\'eaire gradu\'e $V'_\sbullet$, alors on a  (cf. la remarque \ref{Rem:comparaisonmu})
\begin{equation}\label{Equ:compaeps}\varepsilon^{\Theta}(V_\sbullet)\leqslant\varepsilon^{\Theta}(V_\sbullet').\end{equation}
\begin{rema}
Pour tout entier $p\geqslant 1$, soit $V^{(p)}_{\sbullet}$ le syst\`eme lin\'eaire gradu\'e $\bigoplus_{n\geqslant 0}V_{np}$ de $L^{\otimes p}$. La fibre g\'en\'erique de $p_{0*}(V_\sbullet^{(p)})$ s'identifie \`a $W^{(p)}_\sbullet$. On a $\vol(W^{(p)}_\sbullet)=p^d\vol(W^{(p)})$. En outre, on a $\mu_{\max}^{\Theta}(V_\sbullet^{(p)})=p\mu_{\max}^{\Theta}(V_\sbullet)$. On obtient alors de la formule r\'ecursive \eqref{Equ:termeerreur} que
\begin{equation}
\varepsilon^\Theta(V_\sbullet^{(p)})\leqslant p^d\varepsilon^{\Theta}(V_\sbullet)
\end{equation}
\end{rema}
 
\begin{theo}\label{Thm:majorationdeHS}
Soit $X$ un sch\'ema projectif et int\`egre sur $\Spec k$ muni d'un tour de fibrations sur courbes $\Theta=(p_i:X_i\rightarrow C_i)_{i=0}^d$, et $L$ un faisceau inversible sur $X$, o\`u $k$ est un corps de caract\'eristique z\'ero. 
Si $V_\sbullet=\bigoplus_{n\geqslant 0}V_n$ est un syst\`eme lin\'eaire gradu\'e de $L$, qui contient un diviseur ample, alors on a 
\begin{equation}\label{Equ:majoration}\rang_k(V_1)\leqslant\mathrm{vol}(V_\sbullet)+\varepsilon^{\Theta}(V_\sbullet). \end{equation}
\end{theo}
\begin{proof}
Les invariants figurant \`a droite de l'in\'egalit\'e \eqref{Equ:majoration} sont des invariants birationnels. Quitte \`a passer \`a la normalisation de $X$, on peut supposer que $X$ est un sch\'ema normal. On raisonne par r\'ecurrence sur $d$. Dans le cas o\`u $d=0$, le sch\'ema $X$ est une courbe r\'eguli\`ere de genre $g(\Theta)$. Pour tout entier $n\geqslant 0$, soit $L_n$ le sous-$\mathcal O_X$-module de $L^{\otimes n}$ engendr\'e par $V_n$. D'apr\`es le th\'eor\`eme \ref{Thm:passageauxfibres}, $L_\sbullet=\bigoplus_{n\geqslant0}L_n$ est un syst\`eme gradu\'e en fibr\'es vectoriels de $L$, qui contient un diviseur ample et v\'erifie $\vol(L_\sbullet)=\vol(V_\sbullet)$. En outre, l'homomorphisme $L_n\otimes L_m\rightarrow L_{n+m}$ est non-nul d\`es que $L_n$ et $L_m$ sont tous non-nuls. On obtient alors (cf. la remarque \ref{Rem:systemegradue})
\[\vol(V_\sbullet)=\vol(L_\sbullet)=\lim_{n\rightarrow+\infty}\frac{\deg(L_n)}{n}=\sup_{n\geqslant 1}\frac{\deg(L_n)}{n}\] 
En particulier, si $V_1$ est non-nul, alors 
\[\rang_k(V_1)\leqslant h^0(L_1)\leqslant\deg(L_1)+\max(g(X)-1,1)\leqslant\vol(V_\sbullet)+\max(g(X)-1,1).\]
Cette in\'egalit\'e est aussi vraie lorsque $\rang_k(V_1)=\{0\}$. Le th\'eor\`eme est donc d\'emontr\'e pour le cas o\`u $d=0$.

Traitons maintenant le cas g\'en\'eral. Soient 
\[(v_0,\ldots,v_d)=\vol^\Theta(V_\sbullet),\;(\mu_0,\ldots,\mu_d)=\mu_{\max}^{\Theta}(V_\sbullet)\;\text{ et }\;(g_0,\ldots,g_d)=g(\Theta).\] Soient $E_\sbullet:=p_{0*}(V_\sbullet)$ et $W_\sbullet$ la fibre g\'en\'erique de $E_\sbullet$.   D'apr\`es le th\'eor\`eme \ref{Thm:passageauxfibres}, $E_\sbullet=\bigoplus_{n\geqslant 0}E_n$ est un syst\`eme gradu\'e en fibr\'es vectoriels de $L$ qui contient un diviseur ample et v\'erifie la relation $\vol(E_\sbullet)=\vol(V_\sbullet)$. Soit $\eta$ la fibre g\'en\'erique de $C_0$. Pour tout nombre r\'eel $t$ et tout entier $n\in\mathbb N$, soit $W^{t}_n:=(\mathcal F^{nt}E_n)_\eta$, o\`u $\mathcal F$ est la $\mathbb R$-filtration de Harder-Narasimhan sur $E_n$. D'apr\`es le corollaire \ref{Cor:filtrationdeHNenalg}, $W^t_\sbullet=\bigoplus_{n\geqslant 0}W_n^t$ est un syst\`eme lin\'eaire gradu\'e de $L_\eta$. En outre, $W_\sbullet^t$ contient un diviseur ample lorsque $t<\mu_0$, et devient trivial lorsque $t>\mu_0$ (cf. \cite[lemme 1.6]{Boucksom_Chen}), et on a (cf. \cite[corollaire 1.13]{Boucksom_Chen}) 
\begin{equation}\label{Equ:volcommeinte}
\vol(E_\sbullet)=(d+1)\int_0^{\mu_0}\vol(W_\sbullet^t)\,\mathrm{d}t.\end{equation}
Si $E_1$ est un fibr\'e vectoriel non-nul, d'apr\`es la formule \eqref{Equ:degplus} on obtient\footnote{D'apr\`es la remarque \ref{Rem:systemegradue}, on a  $\mu_{\max}(E_1)\leqslant\mu_{\max}^{\mathrm{asy}}(E_\sbullet)=\mu_0$. Donc $W_1^t=\{0\}$ lorsuqe $t>\mu_0$.}
\[\deg_+(E_1)=\int_0^{\mu_0}\rang(W_1^t)\,\mathrm{d}t.\]
On applique l'hypoth\`ese de r\'ecurrence \`a $W_\sbullet^t$ pour chaque $t$ et obtient
\begin{equation}\label{Equ:majoratiodegplus}\deg_+(E_1)\leqslant\int_0^{\mu_0}\Big(\frac{\vol(W_1^t)}{d!}+\varepsilon^{\Theta'}(W_\sbullet^t)\Big)\,\mathrm{d}t\leqslant\frac{\vol(E_\sbullet)}{(d+1)!}+\mu_0\varepsilon^{\Theta'}(W_\sbullet),\end{equation}
o\`u $\Theta'=(p_i:X_i\rightarrow C_i)_{i=1}^d$. Dans la deuxi\`eme in\'egalit\'e, on a utilis\'e les relations \eqref{Equ:volcommeinte} et \eqref{Equ:compaeps}.
Enfin, le th\'eor\`eme \ref{Thm:h0etdegplus}  montre que
\[\begin{split}h^0(E_1)&\leqslant\deg_+(E_1)+\rang(W_1)\max(g_0-1,1)\\
&\leqslant\frac{\vol(E_\sbullet)}{(d+1)!}+\mu_0\varepsilon^{\Theta'}(W_\sbullet)+\Big(\frac{v_1}{d!}+\varepsilon^{\Theta'}(W_\sbullet^{\Theta'})\Big)\max(g_0-1,1),
\end{split}\]
o\`u on a appliqu\'e l'hypoth\`ese de r\'ecurrence \`a $W_\sbullet$  dans la deuxi\`eme in\'egalit\'e.
Comme $\rang(V_1)\leqslant h^0(E_1)$ et comme $\vol(E_\sbullet)=\vol(V_\sbullet)$, on obtient
\[\rang(V_1)\leqslant\frac{\vol(V_\sbullet)}{(d+1)!}+\mu_0\varepsilon^{\Theta'}(W_\sbullet)+\Big(\frac{v_1}{d!}+\varepsilon^{\Theta'}(W_\sbullet)\Big)\max(g_0-1,1).
\]
La d\'emonstration est donc achev\'ee, compte tenu de la formule r\'ecursive \eqref{Equ:termeerreur} d\'efinissant $\varepsilon^{\Theta}(V_\sbullet)$.
\end{proof}

Dans la d\'emonstration de l'in\'egalit\'e \eqref{Equ:majoratiodegplus}, on a seulement utilis\'e le fait que l'alg\`ebre gradu\'ee $W_\sbullet$ est filtr\'ee par la $\mathbb R$-filtration de Harder-Narasimhan. Ainsi le th\'eor\`eme se g\'en\'eralise naturellement dans le cadre de syst\`eme lin\'eaire gradu\'e filtr\'e et conduit au corollaire suivant qui sera utile plus loin dans l'\'etude des syst\`emes lin\'eaires gradu\'es arithm\'etiques.

\begin{coro}\label{Cor:majHS}
Soit $X$ un sch\'ema projectif et int\`egre sur $\Spec k$ muni d'un tour de fibrations sur courbes $\Theta$. Soient $L$ un faisceau inversible gros sur $X$ et $V_\sbullet$ un syst\`eme lin\'eaire gradu\'e de $L$ qui contient un diviseur ample. On suppose que chaque espace vectoriel $V_n$ est muni d'une $\mathbb R$-filtration $\mathcal F$ de sorte que
\begin{equation}\label{Equ:sousmult}(\mathcal F^aV_n)\cdot(\mathcal F^bV_m)\subset \mathcal F^{a+b}V_{n+m}\end{equation}
pour tout $(a,b)\in\mathbb R^2$ et tout $(n,m)\in\mathbb N^2$. Soit en outre 
\[\lambda_{\max}^{\mathrm{asy}}(V_\sbullet)=\sup_{n\geqslant 1}\frac{\sup\{t\,|\,\mathcal F^tV_n\neq\{0\}\}}{n}.\]
Alors on a 
\begin{equation}
\int_0^{+\infty}\rang_k(\mathcal F^tV_1)\,\mathrm{d}t\leqslant\int_0^{+\infty}\frac{\vol(V^t_\sbullet)}{d!}\,\mathrm{d}t+\lambda_{\max}^{\mathrm{asy}}(V_\sbullet)\varepsilon^{\Theta}(V_\sbullet),
\end{equation}
o\`u $d=\dim(X)$ et $V_\sbullet^t=\bigoplus_{n\geqslant 0}\mathcal F^{nt}V_n$. 
\end{coro}
\begin{proof}
Pour tout entier $n\geqslant 1$, on note \[\lambda_{\max}(V_n)=\sup\{t\,|\,\mathcal F^tV_n\neq\{0\}\}.\]
La condition \eqref{Equ:sousmult} montre que la suite $(\lambda_{\max}(V_n))_{n\geqslant 1}$ est sur-additive, d'o\`u $\lambda_{\max}(V_n)\leqslant n\lambda_{\max}^{\mathrm{asy}}(V_\sbullet)$. Par cons\'equent, pour tout entier $n\geqslant 1$ et tout $t>\lambda_{\max}^{\mathrm{asy}}(V_\sbullet)$, on a $V_n^t=\mathcal F^{nt}V_n=\{0\}$. En outre, d'apr\`es \cite[lemme 1.6]{Boucksom_Chen}, pour tout $t<\lambda_{\max}^{\mathrm{asy}}(V_\sbullet)$, le syst\`eme lin\'eaire gradu\'e $V_\sbullet^t$ contient un diviseur ample. On applique le th\'eor\`eme \ref{Thm:majorationdeHS} \`a $V_\sbullet^t$ et obtient 
\[\rang_k(\mathcal F^tV_1)\leqslant\frac{\vol(V_\sbullet^t)}{d!}\,\mathrm{d}t+\varepsilon^{\Theta}(V_\sbullet^t)\leqslant \frac{\vol(V_\sbullet^t)}{d!}\,\mathrm{d}t+\varepsilon^{\Theta}(V_\sbullet),\quad t<\lambda_{\max}^{\mathrm{asy}}(V_\cdot).\]
L'int\'egration de cette in\'egalit\'e pour $t\in [0,\lambda_{\max}^{\mathrm{asy}}(V_\sbullet)[$ conduit \`a l'in\'egalit\'e souhait\'ee.
\end{proof}

\section{Fibr\'es vectoriels ad\'eliques}

Dans ce paragraphe, l'expression $K$ d\'esigne un corps de nombres. Soit $\mathcal O_K$ la fermeture int\'egrale de $\mathbb Z$ dans $K$. On entend par \emph{place} de $K$ toute classe d'\'equivalence de valeurs absolues non-triviales sur $K$, o\`u deux valeurs absolues sont dites \'equivalentes si elles d\'efinissent la m\^eme topologie sur $K$. On d\'esigne par $M_K$ l'ensemble de toutes les places de $K$.  Pour toute place $v$ de $K$, on d\'esigne par $|.|_v$ une valeur absolue dans la place $v$ qui prolonge soit la valeur absolue usuelle sur $\mathbb Q$, soit l'une des valeurs absolues $p$-adiques.  On d\'esigne par $K_v$ le compl\'et\'e de $K$ par rapport \`a la valeur absolue $|.|_v$. Il s'av\`ere que $|.|_v$ s'\'etend de fa\c{c}on unique sur la cl\^oture alg\'ebrique $\overline K_v$. On d\'esigne par $\mathbb C_v$ le compl\'et\'e de $\overline K_v$, qui est \`a la fois alg\'ebriquement clos et complet. On rappelle que la famille $(|.|_v)_{v\in M_K}$ de valeurs absolues v\'erifie la formule du produit
\begin{equation}\label{Equ:formuleduproduit}
\forall\, a\in K^{\times},\; \sum_{v\in M_K}\frac{[K_v:\mathbb Q_{v}]}{[K:\mathbb Q]}\ln |a|_v=0.
\end{equation}  
Il s'av\`ere que l'ensemble des places de $K$ qui ne prolongent pas $\infty\in M_{\mathbb Q}$ (repr\'esentant la valeur absolue usuelle de $\mathbb Q$) correspond biunivoquement \`a l'ensemble des points ferm\'es de $\Spec\mathcal O_K$.

Soit $E$ un espace vectoriel de rang fini sur $K$. Si $\boldsymbol{e}=(e_1,\ldots,e_n)$ est une base de $E$, alors elle d\'efinit pour chaque place $v\in M_K$ une norme $\|.\|_{\boldsymbol{e},v}$ sur $E\otimes_K\mathbb C_v$ telle que
\[\|\lambda_1e_1+\cdots+\lambda_ne_n\|_{\boldsymbol{e},v}=\begin{cases}
\max\{|\lambda_1|_v,\ldots,|\lambda_n|_v\},&\text{si $v$ est non-archim\'edienne},\\
(|\lambda_1|_v^2+\cdots+|\lambda_n|_v^2)^{1/2},&\text{si $v$ est archim\'edienne}.
\end{cases}\]
Cette norme est invariante sous l'action du groupe de Galois $\mathrm{Gal}(\mathbb C_v/K_v)$.

\begin{defi}
On appelle \emph{fibr\'e vectoriel ad\'elique} (cf. \cite{Gaudron08}) sur $\Spec K$ toute donn\'ee $\overline E=(E,(\|.\|_{v}))$ d'un espace vectoriel de rang fini $E$ sur $K$ et une famille de normes, o\`u $\|.\|_v$ est une norme sur $E\otimes_K\mathbb C_v$, qui est invariante sous l'action du groupe de Galois $\mathrm{Gal}(\mathbb C_v/K_v)$, et est ultram\'etrique lorsque $v$ est une place non-archim\'edienne. On demande en plus l'existence d'une base $\boldsymbol{e}$ de $E$ telle que $\|.\|_v=\|.\|_{\boldsymbol{e},v}$ pour toute sauf un nombre fini de places $v\in M_K$. Si $E$ est de rang $1$ sur $K$, on dit que $\overline E$ est un \emph{fibr\'e inversible ad\'elique}. Si, pour toute place archim\'edienne $v$, la norme $\|.\|_v$ est hermitienne, on dit que $\overline E$ est \emph{hermitien}.
\end{defi}

Soit $\overline L$ un fibr\'e inversible ad\'elique sur $\Spec K$. On d\'efinit le \emph{degr\'e d'Arakelov} de $\overline L$ comme
\begin{equation}
\hdeg(\overline L):=-\sum_{v\in M_K}[K_v:\mathbb Q_{v}]\ln\|s\|_v,
\end{equation}
o\`u $s$ est un \'el\'ement non-nul de $L$. D'apr\`es la formule du produit \eqref{Equ:formuleduproduit}, cette d\'efinition ne d\'epend pas du choix de $s$. On introduit aussi la version normalis\'ee du degr\'e d'Arakelov comme
\begin{equation}
\ndeg(\overline L):=\frac{\deg(L)}{[K:\mathbb Q]}.
\end{equation}

La notion de \emph{fibr\'e vectoriel norm\'e} (ou \emph{hermitien}) sur $\Spec\mathcal O_K$ est aussi largement utilis\'ee dans la litt\'erature. Rappelons qu'un fibr\'e vectoriel norm\'e sur $\Spec\mathcal O_K$ est la donn\'ee $\overline{\mathcal E}$ d'un $\mathcal O_K$-module projectif et de type fini muni d'une famille de norme $(\|.\|_v)_{v\in M_{K,\infty}}$ index\'ee par l'ensemble des places archim\'ediennes de $K$, o\`u chaque $\|.\|_{v}$ est une norme sur $\mathcal E\otimes_{\mathcal O_K}\mathbb C_v$, invariante sous l'action du groupe de Galois $\mathrm{Gal}(\mathbb C/K_v)$. Le fibr\'e vectoriel norm\'e $\overline{\mathcal E}$ est dit hermitien si chaque norme $\|.\|_v$ est hermitienne. \'Etant donn\'e un fibr\'e vectoriel norm\'e $\overline{\mathcal E}$ sur $\Spec\mathcal O_K$, on obtient naturellement une structure de fibr\'e vectoriel ad\'elique pour $E=\mathcal E_K$ (qui est hermitienne lorsque $\overline{\mathcal E}$ est hermitien), o\`u la norme en une place non-archim\'edienne $\mathfrak p$ est induite par la structure de $\mathcal O_K$-module de $\mathcal E$~: la boule unit\'e ferm\'ee de $\|.\|_{\mathfrak p}$ est $\mathcal E\otimes_{\mathcal O_K}\mathcal O_{\mathfrak p}$, o\`u $\mathcal O_{\mathfrak p}$ est l'anneau de valuation de $\mathbb C_{\mathfrak p}$. R\'eciproquement, un fibr\'e vectoriel ad\'elique sur $\Spec K$ provient n\'ecessairement d'un fibr\'e vectoriel norm\'e sur $\Spec\mathcal O_K$ pourvu que toutes ses normes index\'ees par les places non-archim\'ediennes sont pures\footnote{Soit $\overline E=(E,(\|.\|_v)_{v\in M_K})$ un fibr\'e vectoriel ad\'elique sur $\Spec K$. Pour toute place non-archim\'edienne $\mathfrak p$, la norme $\|.\|_{\mathfrak p}$ est dite pure si l'image de sa restriction \`a $E$ s'identifie \`a l'image de la valeur absolue $|.|_{\mathfrak p}:K\rightarrow\mathbb R$.}. On renvoie les lecteurs dans \cite[proposition 3.10]{Gaudron09} pour une d\'emonstration.

\subsection{Filtration de Harder-Narasimhan}
Soit $\overline E=(E,(\|.\|_v)_{v\in M_K})$ un fibr\'e vectoriel ad\'elique sur $\Spec K$ qui est hermitien. Pour toute place $v\in M_K$, la norme $\|.\|_v$ sur $E_{\mathbb C_v}$ induit une norme sur $\det(E_{\mathbb C_v})$ qui est une ultranorme (resp. une norme hermitienne) lorsque $v$ est une place ultram\'etrique (resp. une place archim\'edienne), et telle que 
\[\forall\,(s_1,\ldots,s_r)\in E^r,\quad\|s_1\wedge\cdots\wedge s_r\|_v=\prod_{i=1}^r\|s_i\|_v,\]
o\`u $r$ est le rang de $E$. Ainsi $\det(\overline E):=(\det E,(\|.\|_v)_{v\in M_K})$ devient un fibr\'e inversible ad\'elique sur $\Spec K$. On d\'efinit le \emph{degr\'e d'Arakelov normalis\'e} de $\overline E$ comme celui de $\det(\overline E)$, not\'e comme $\ndeg(\overline E)$. Si de plus $E$ est non-nul, on d\'esigne par $\hmu(E)$ le quotient $\hdeg_n(\overline E)/\rang(E)$, appel\'e la \emph{pente} de $\overline E$. Le formalisme de la th\'eorie de Harder-Narasimhan est encore valable dans le cadre des fibr\'es vectoriels ad\'eliques hermitiens. En particulier, il existe un sous-espace vectoriel $E_{\des}$ de $E$ tel que 
\[\hmu(\overline E_{\des})=\hmu_{\max}(\overline E):=\sup_{0\neq F\subset E}\hmu(\overline F)\]
et que $E_{\des}$ contient tous les sous-espaces vectoriels non-nuls $F$ de $E$ tel que $\overline F$ admet $\hmu_{\max}(\overline E)$ comme pente. On renvoie les lecteurs dans \cite[\S5.1]{Gaudron08} pour les d\'etails. Similairement au cas de fibr\'es vectoriels sur une courbe, on peut construire de fa\c{c}on r\'ecursive un drapeau 
\[0=E_0\subsetneq E_1\subsetneq\ldots\subsetneq E_n=E\]
de sous-espaces vectoriels de $E$ de sorte que $E_i/E_{i-1}=(E/E_{i-1})_{\des}$ pour tout $i\in\{1,\ldots,n\}$, o\`u on a consid\'er\'e des normes quotients sur $E/E_{i-1}$. Si on d\'esigne par $\alpha_i$ la pente de $\overline E_i/\overline E_{i-1}$, alors on a 
\[\alpha_1>\alpha_2>\ldots>\alpha_n.\]
La derni\`ere pente $\alpha_n$ est appel\'ee la \emph{pente minimale} de $\overline E$, not\'ee comme $\hmu_{\min}(\overline E)$. On d\'esigne par $P_{\overline E}$ la fonction concave et affine par morceau d\'efinie sur l'intervalle $[0,\rang(E)]$, qui est affine sur chaque intervalle $[\rang(E_{i-1}),\rang(E_i)]$ et de pente $\alpha_i$. Cette fonction est appel\'ee le \emph{polygone de Harder-Narasimhan} de $\overline E$.
On d\'esigne par $\mathcal F_{\mathrm{HN}}$ la $\mathbb R$-filtration d\'ecroissante sur $E$ telle que
\begin{equation}
\mathcal F_{\mathrm{HN}}^t(E)=\sum_{\begin{subarray}{c}1\leqslant i\leqslant n\\
t\leqslant\alpha_i\end{subarray}}E_i.
\end{equation}
Cette filtration est appel\'ee la \emph{filtration de Harder-Narasimhan} de $\overline E$.

\begin{defi}\label{Def:degplus}
Soit $\overline E$ un fibr\'e vectoriel ad\'elique hermitien sur $\Spec K$. Soit $r$ le rang de $E$. Pour tout $i\in\{1,\ldots,r\}$, on d\'esigne par $\hmu_i(\overline E)$ la pente du polygone de Harder-Narasimhan $P_{\overline E}$ sur l'intervalle $[i-1,i]$, appel\'ee la \emph{$i^{\text{\`eme}}$ pente} de $\overline E$. En outre, on d\'esigne par $\npdeg(\overline E)$ la valeur maximale du polygone $P_{\overline E}$ sur l'intervalle $[0,r]$.
\end{defi}

\begin{prop}
Avec les notation de la d\'efinition pr\'ec\'edente, on a 
\begin{equation}\label{Equ:deg+}\npdeg(\overline E)=\sum_{\begin{subarray}{c}
1\leqslant i\leqslant r\\
\widehat{\mu}_i(\overline E)\geqslant 0
\end{subarray}}\hmu_i(\overline E)=\int_0^{+\infty}\rang(\mathcal F^t_{\mathrm{HN}}(E))\,\mathrm{d}t.\end{equation}
\end{prop}
\begin{proof}
Comme la fonction $P_{\overline E}$ est affine sur chaque intervalle $[i-1,i]$, on obtient que 
\[\npdeg(\overline E)=\max_{i\in\{0,\ldots,r\}}P_{\overline E}(i)=\max_{i\in\{0,\ldots,r\}}\sum_{1\leqslant j\leqslant i}\hmu_j(\overline E).\]
En outre, comme la fonction $P_{\overline E}$ est concave, on a $\hmu_1(\overline E)\geqslant\ldots\geqslant\hmu_r(\overline E)$. La premi\`ere \'egalit\'e de \eqref{Equ:deg+} est donc d\'emontr\'ee.

Soit $0=E_0\subsetneq E_1\subsetneq\ldots\subsetneq E_n=E$ le drapeau de Harder-Narasimhan de $\overline E$. Pour tout $j\in\{1,\ldots,n\}$, soit $\alpha_j$ la pente de la fonction $P_{\overline E}$ sur l'intervalle $[\rang(E_{j-1}),\rang(E_j)]$. Par la d\'efinition de $\mathcal F^t_{\mathrm{HN}}$, on a
\[\mathrm{d}\rang(\mathcal F_{\mathrm{HN}}^t(E))=-\sum_{j=1}^n\big(\rang(E_j)-\rang(E_{j-1})\big)\delta_{\alpha_j}=-\sum_{i=1}^r\delta_{\hmu_i(\overline E)}\]
comme mesures bor\'eliennes sur $\mathbb R$, o\`u $\delta_x$ d\'esigne la mesure de Dirac en $x$. On obtient alors
\[\int_0^{+\infty}\rang(\mathcal F_{\mathrm{HN}}^t(E))\,\mathrm{d}t=-\int_{[0,+\infty[}t\,\mathrm{d}\rang(\mathcal F_{\mathrm{HN}}^t(E))=\sum_{\begin{subarray}{c}
1\leqslant i\leqslant r\\
\widehat{\mu}_i(\overline E)\geqslant 0
\end{subarray}}\hmu_i(\overline E).\]
\end{proof}

Le r\'esultat suivant relie la fonction $\npdeg(.)$ au nombre d'\'el\'ements effectifs dans un fibr\'e vectoriel ad\'elique. Pour tout fibr\'e vectoriel ad\'elique $\overline E$ sur $\Spec K$, on d\'esigne par $\widehat{H}^0(\overline E)$ l'ensemble des \'el\'ements $s\in E$ tels que $\sup_{v\in M_K}\|s\|_v\leqslant 1$ (un tel \'el\'ement est appel\'e une \emph{section effective} de $\overline E$). C'est un ensemble fini. On d\'esigne par $\widehat{h}^0(\overline E)$ le nombre r\'eel $\ln(\#\widehat{H}^0(\overline E))$. Rappelons d'abord un r\'esultat de Gillet et Soul\'e \cite[th\'eor\`eme 2]{Gillet-Soule91}. Pour tout entier $n\geqslant 1$, on introduit une constante $C(K,n)$ comme la suite 
\[nd_K\ln(3)+n(r_1+r_2)\ln(2)+\frac{n}{2}\ln|\Delta_K|-r_1\ln(V(B_n)n!)-r_2\ln(V(B_{2n})(2n)!)+\ln((d_Kn)!),\]
o\`u $d_K=[K:\mathbb Q]$, $\Delta_K$ est le discriminant du corps $K$, $B_n$ d\'esigne la boule unit\'e dans $\mathbb R^n$, $V(.)$ est la mesure de Lebesgue, et $r_1$ et $r_2$ sont respectivement le nombre des places r\'eelles et complexes de $K$. Rappelons que la formule de Sterling \[n!\sim \sqrt{2\pi n}(n/e)^n\quad (n\rightarrow+\infty)\] et la relation 
\[V(B_n)=\frac{\pi^{n/2}}{\Gamma(n/2+1)}\]
montrent que
\[C(K,n)=\frac 12[K:\mathbb Q]n\ln(n)+O(n),\qquad n\rightarrow+\infty.\]
 
\begin{theo}[Gillet-Soul\'e]
Soit $\overline E$ un fibr\'e vectoriel ad\'elique sur $\Spec K$ qui provient d'un fibr\'e vectoriel hermitien sur $\Spec\mathcal O_K$. Soit $n$ le rang de $E$ sur $K$. Alors  on a
\begin{equation}\label{Equ:Gillet-Soule}\Big|\widehat{h}^0(\overline E)-\widehat{h}^0(\overline\omega_{K/\mathbb Q}\otimes \overline E^\vee)-
\hdeg(\overline E)\Big|\leqslant C(K,n),\end{equation}
o\`u $\overline\omega_{K/\mathbb Q}$ est le fibr\'e inversible ad\'elique associ\'e \`a $\omega_{\mathcal O_K}=\Hom_{\mathbb Z}(\mathcal O_K,\mathbb Z)$ muni de la famille de normes $(\|.\|_v)_{v\in M_{K,\infty}}$ telles que $\|\mathrm{tr}_{K/\mathbb Q}\|_v=1$ pour tout $v\in M_{K,\infty}$.
\end{theo}

Le fibr\'e inversible ad\'elique $\overline{\omega}_{K/\mathbb Q}$ devrait \^etre consid\'er\'e comme le faisceau dualisant relative arithm\'etique de $\Spec K\rightarrow\Spec\mathbb Q$. Son degr\'e d'Arakelov est $\ln|\Delta_K|$. 

\begin{lemm}\label{Lem:estdeh0}
Soit $\overline E$ un fibr\'e vectoriel ad\'elique non-nul sur $\Spec K$ qui provient d'un fibr\'e vectoriel hermitien sur $\Spec\mathcal O_K$.
\begin{enumerate}[(a)]
\item Si $\hmu_{\max}(\overline E)<0$, alors $\widehat{h}^0(\overline E)=0$.
\item Si $\hmu_{\min}(\overline E)>[K:\mathbb Q]^{-1}\ln|\Delta_K|$, alors $|\widehat{h}^0(\overline E)-\hdeg(\overline E)|\leqslant C(K,\rang(E))$.
\item Si $\hmu_{\min}(\overline E)\geqslant 0$, alors $|\widehat{h}^0(\overline E)-\hdeg(\overline E)|\leqslant\ln|\Delta_K|\rang(E)+C(K,\rang(E))$.
\end{enumerate}
\end{lemm}
\begin{proof}
Pour tout nombre r\'eel $t$, on d\'esigne par $\mathcal O(t)$ le fibr\'e inversible ad\'elique sur $\Spec K$ dont l'espace vectoriel sous-jacent s'identifie \`a $K$ et telle que $\|1\|_{\mathfrak p}=1$ pour toute place non-archim\'edienne $\mathfrak p$  et que $\|1\|_{\sigma}=\mathrm{e}^{-1}$ pour toute place archim\'edienne $\sigma$. Le degr\'e d'Arakelov du fibr\'e inversible ad\'elique $\mathcal O(t)$ est alors $[K:\mathbb Q]t$.

(a) On suppose que $\overline E$ poss\`ede une section effective non-nulle, qui d\'efinit un homomorphisme injectif de $\mathcal O(0)$ vers $\overline E$. D'apr\`es l'in\'egalit\'e de pentes (cf. \cite[lemme 6.4]{Gaudron08}), on obtient $\hmu_{\max}(\overline E)\geqslant 0$.

(b) Comme $\hmu_{\min}(\overline E)>[K:\mathbb Q]^{-1}\ln|\Delta_K|$, on a \[\hmu_{\max}(\overline{\omega}_{K/\mathbb Q}\otimes\overline E{}^\vee)=\ndeg(\overline{\omega}_{K/\mathbb Q})+\hmu_{\max}(\overline E{}^\vee)=[K:\mathbb Q]^{-1}\ln|\Delta_K|-\hmu_{\min}(\overline E)<0.\] Par le r\'esultat obtenu dans (a), on obtient $\widehat{h}^0(\overline{\omega}_{K/\mathbb Q}\otimes\overline E^\vee)=0$. L'in\'egalit\'e annonc\'ee provient donc de la formule \eqref{Equ:Gillet-Soule}.

(c) Soient $\varepsilon>0$ un nombre r\'eel et $t=[K:\mathbb Q]^{-1}\ln|\Delta_K|+\varepsilon$. On a alors $\hmu_{\min}(\overline E\otimes\mathcal O(t))>[K:\mathbb Q]^{-1}\ln|\Delta_K|$. D'apr\`es (b), on obtient
\[\widehat{h}^0(\overline E\otimes\mathcal O(t))\leqslant\hdeg(\overline E\otimes\mathcal O(t))+C(K,\rang(E))=\hdeg(\overline E)+t\rang(E)+C(K,\rang(E)).\]
En outre, comme $t>0$, on a $\widehat{h}^0(\overline E)\leqslant\widehat{h}^0(\overline E\otimes\mathcal O(t))$. On obtient donc 
\[\widehat{h}^0(\overline E)-\hdeg(\overline E)\leqslant t\rang(E)+C(K,\rang(E)).\]
De plus, l'in\'egalit\'e \eqref{Equ:Gillet-Soule} implique que
\[\widehat{h}^0(\overline E)-\hdeg(\overline E)\geqslant\widehat{h}^0(\overline{\omega}_{K/\mathbb Q}\otimes\overline E^\vee)-C(K,\rang(E))\geqslant -C(K,\rang(E)).\]
On obtient alors
\[|\widehat{h}^0(\overline E)-\hdeg(\overline E)|\leqslant t\rang(E)+C(K,\rang(E)).\]
Comme $\varepsilon$ est arbitraire, on obtient le r\'esultat souhait\'e.
\end{proof}

\begin{theo}
Soit $\overline E$ un fibr\'e vectoriel ad\'elique non-nul sur $\Spec K$ qui provient d'un fibr\'e vectoriel hermitien sur $\Spec\mathcal O_K$. Alors on a 
\begin{equation}\label{Equ:h0degset}
\big|\widehat{h}^0(\overline E)-[K:\mathbb Q]\npdeg(\overline E)\big|\leqslant\rang(E)\ln|\Delta_K|+C(K,\rang(E)).
\end{equation}
\end{theo}
\begin{proof}
Soient $0=E_0\subsetneq E_1\subsetneq\ldots\subsetneq E_n=E$ le drapeau de Harder-Narasimhan de $E$, et $\alpha_i=\hmu(\overline E_i/\overline E_{i-1})$. Soit $j$  le dernier indice dans $\{1,\ldots,n\}$ tel que $\alpha_j\geqslant 0$. Si un tel indice n'existe pas, on prend $j=0$ par convention. Par d\'efinition on a $\npdeg(\overline E)=\ndeg(\overline E_j)$. Si $j=0$, alors $\hmu_{\max}(\overline E)=0$, et donc $\widehat{h}^0(\overline E)=0$. En outre, par convention on a $\npdeg(\overline E)=0$. L'in\'egalit\'e \eqref{Equ:h0degset} est donc triviale.
Si $j>0$, alors $\hmu_{\max}(\overline E/\overline E_j)=\alpha_{j-1}<0$. Par cons\'equent, on a  $\widehat{h}^0(\overline E/\overline E_j)=0$ et donc $\widehat{h}^0(\overline E)=\widehat{h}^0(\overline E_j)$. En outre, on a $\hmu_{\min}(\overline E_j)=\alpha_j\geqslant 0$, et par le lemme \ref{Lem:estdeh0}.(c), on obtient
\[\begin{split}&\quad\;\big|\widehat{h}^0(\overline E)-[K:\mathbb Q]\npdeg(\overline E)\big|=\big|\widehat{h}^0(\overline E_j)-[K:\mathbb Q]\npdeg(\overline E_j)\big|\\
&\leqslant\rang(E_j)\ln|\Delta_K|+C(K,\rang(E_j))\leqslant\rang(E)\ln|\Delta_K|+C(K,\rang(E)).
\end{split}\] 
Le r\'esultat est donc d\'emontr\'e.
\end{proof}

\subsection{Filtration par hauteur}
Soit $\overline E=(E,(\|.\|_v)_{v\in M_K})$ un fibr\'e vectoriel ad\'elique sur $\Spec K$ (qui n'est pas n\'ecessairement hermitien). Pour tout \'el\'ement non-nul $s$ de $\overline E$, on d\'esigne par $h_{\overline E}(s)$ le nombre $-\ndeg(\overline{Ks})$, appel\'e la \emph{hauteur normalis\'ee} de $s$. Plus g\'en\'eralement, si $K'$ est une extension finie de $K$, on peut construire un fibr\'e vectoriel ad\'elique $\overline E\otimes_KK'$ sur $\Spec K'$ dont l'espace vectoriel sous-jacent est $E_{K'}:=E\otimes_KK'$ et dont la norme en $v'\in M_{K'}$ s'identifie \`a $\|.\|_v$ avec $v\in M_K$, o\`u $v'$ prolonge $v$. Ainsi on peut d\'efinir la hauteur normalis\'ee pour tout \'el\'ement non-nul de $E_{K'}$. En outre, pour toute extension finie $K''$ de $K'$, la hauteur normalis\'ee de $s\in E_{K'}$ s'identifie \`a celle de son image canonique dans $E_{K''}$. Cette observation permet d'\'etendre $h_{\overline E}$ en une fonction sur l'ensemble des vecteurs non-nuls dans $E_{K^{\mathrm{a}}}$, o\`u $K^{\mathrm{a}}$ d\'esigne la cl\^oture alg\'ebrique de $K$. En outre, d'apr\`es la formule du produit, la fonction $h_{\overline E}$ est invariante sous la multiplication par un scalaire non-nul dans $K^{\mathrm{a}}$. Ainsi on peut la consid\'erer comme une fonction d\'efinie sur l'ensemble des points alg\'ebriques de $\mathbb P(E^\vee)$. Il s'av\`ere que cette fonction s'identifie \`a la hauteur absolue par rapport au faisceau inversible universel $\mathcal O_{E^\vee}(1)$ muni des m\'etriques de Fubini-Study.

Pour tout nombre r\'eel $t$, on d\'esigne par $\mathcal F_{\mathrm{ht}}^t(E_{K^{\mathrm{a}}})$ le sous-$K^{\mathrm{a}}$-espace vectoriel de $E_{K^{\mathrm{a}}}$ engendr\'e par tous les vecteurs non-nuls $s$ v\'erifiant $h_{\overline E}(s)\leqslant -t$. La famille $(\mathcal F_{\mathrm{ht}}^t(E_{K^{\mathrm{a}}}))_{t\in\mathbb R}$ d\'efinit une $\mathbb R$-filtration d\'ecroissante de l'espace vectoriel $E_{K^{\mathrm{a}}}$. Ses points de saut successifs (o\`u on compte les multiplicit\'es) sont la version logarithmique des minima absolus d\'efinis par Roy et Thunder \cite{Roy_Thunder96} ({\huayi voir aussi \cite{Gaudron_Remond13} pour une pr\'esentation d\'etaill\'ee de diff\'erentes notions de minima}) dans le cadre de la g\'eom\'etrie des nombres ad\'eliques, et par Soul\'e dans le cadre de la g\'eom\'etrie d'Arakelov\footnote{Dans son expos\'e au colloque ``\emph{Arakelov theory and its arithmetic applications.}'' le 22 f\'evrier 2010 \`a Regensbourg}. Pour tout entier $i\in\{1,\ldots,\rang(E)\}$, on d\'esigne par $\Lambda_i(\overline E)$ le nombre
\[\sup\{t\in\mathbb R\,:\,\rang_{K^{\mathrm{a}}}(\mathcal F^t_{\mathrm{ht}}(E_{K^{\mathrm{a}}}))\geqslant i\},\]
appel\'e le $i^{\text{\`eme}}$ minimum absolu logarithmique de $\overline E$.

La comparaison entre les minima successifs et les pentes successives d'un fibr\'e vectoriel ad\'elique hermitien est un probl\`eme naturel. Les in\'egalit\'es $\Lambda_i(.)\leqslant\widehat{\mu}_i(.)$ sont relativement standards et r\'esultent de l'in\'egalit\'e de pentes. Cependant la comparaison au sens inverse est beaucoup plus d\'elicate. Conjecturalement {\huayi on a
\begin{equation}\label{Equ:compainv}\mu_i(\overline E)\leqslant\Lambda_i(\overline E)+\frac 12\ln(\rang(E)),\quad i\in\{1,\ldots,\rang(E)\}\end{equation}}
pour tout fibr\'e vectoriel ad\'elique hermitien $\overline E$ sur $\Spec K$. Une approche possible pour attaquer ce probl\`eme est d'\'etablir une version absolue du th\'eor\`eme de transf\'erence \`a la Banaszczyk \cite{Banaszczyk95}, qui n'est malheureusement pas encore disponible.

Dans la suite, on \'etablit une version plus faible de l'in\'egalit\'e \eqref{Equ:compainv}. Il s'agit d'une comparaison explicite entre 
\[\sum_{i=1}^{\rang(E)}\max(\hmu_i(\overline E),0)\quad\text{et}\quad\sum_{i=1}^{\rang(E)}\max(\Lambda_i(\overline E),0)\]
qui provient du lemme de Siegel absolu d\^u \`a Bombieri-Vaaler \cite{Bombieri_Vaaler83} et Zhang \cite{Zhang95}.

\begin{prop}
Si $\overline E$ est un fibr\'e vectoriel ad\'elique hermitien sur $\Spec K$, alors on a
\begin{equation}\label{Equ:Siegeltronque}\sum_{i=1}^{\rang(E)}\max(\hmu_i(\overline E),0)\leqslant\sum_{i=1}^{\rang(E)}\max(\Lambda_i(\overline E),0)+\frac 12\rang(E)\ln(\rang(E)).\end{equation}
\end{prop}
\begin{proof}
Rappelons que le lemme de Siegel absolu montre que, pour tout fibr\'e vectoriel ad\'elique hermitien $\overline E$ sur $\Spec K$, on a (cf. \cite[th\'eor\`eme 4.14]{Gaudron08} et \cite[\S2.1.3]{Gaudron09} \footnote{Comme on consid\`ere les minima absolus, le d\'efaut de puret\'e est anodin ici.})
\begin{equation}\label{Equ:sigeltronque2}\sum_{i=1}^{\rang(E)}\Lambda_i(\overline E)\geqslant \ndeg(\overline E)-\frac{1}{2}\rang(E)\ln(\rang(E)).\end{equation}
L'in\'egalit\'e est donc vraie lorsque $\hmu_{\min}(\overline E)\geqslant 0$. 

Dans le cas g\'en\'eral, il existe un sous-espace vectoriel $F$ de $E$ tel que $\widehat{\mu}_{\min}(\overline F)\geqslant 0$ et que 
\begin{equation}\label{Equ:mufetmue}\ndeg(\overline F)=\sum_{i=1}^{\rang(F)}\widehat{\mu}_i(\overline F)=\sum_{i=1}^{\rang(E)}\max(\widehat{\mu}_i(\overline E),0).\end{equation}
On peut choisir $F$ comme le dernier sous-espace vectoriel dans le drapeau de Harder-Narasimhan de $E$ v\'erifiant $\widehat{\mu}_{\min}(\overline F)\geqslant 0$. Comme $F$ est un sous-espace vectoriel de $E$, on a $\Lambda_i(\overline F)\leqslant\Lambda_i(\overline E)$ pour tout $i\in\{1,\ldots,\rang(F)\}$. L'in\'egalit\'e \eqref{Equ:sigeltronque2} appliqu\'ee \`a $\overline F$ montre alors que
\[\ndeg(F)
\leqslant\sum_{i=1}^{\rang(F)}\Lambda_i(\overline F)+\frac 12\rang(F)\ln(\rang(F))\leqslant\sum_{i=1}^{\rang(F)}\Lambda_i(\overline E)+\frac 12\rang(E)\ln(\rang(E)).\]
D'apr\`es la formule \eqref{Equ:mufetmue}, on en d\'eduit l'in\'egalit\'e \eqref{Equ:Siegeltronque}.
\end{proof}

\begin{rema} {\huayi L'in\'egalit\'e \eqref{Equ:Siegeltronque} est une cons\'equence imm\'ediate de \eqref{Equ:compainv}. En outre,}
en utilisant le lemme de Siegel absolu, on peut montrer que, si $\overline E$ est un fibr\'e vectoriel ad\'elique sur $\Spec K$ qui est hermitien, alors on a 
\begin{equation}
\label{Equ:mu1etlambda1}
\widehat{\mu}_1(\overline E)\leqslant\Lambda_1(\overline E)+\frac12\ln(\rang(E)).
\end{equation}
On renvoie les lecteurs dans \cite[\S3.2]{Gaudron_Remond} pour une d\'emonstration.
\end{rema}

\subsection{Le cas non-hermitien}
Soit $\overline E$ un  fibr\'e vectoriel ad\'elique g\'en\'eral. Gaudron a d\'efini dans \cite{Gaudron08} le degr\'e d'Arakelov\footnote{Si on fixe un isomorphisme d'espaces vectoriels $\phi:E\rightarrow K^n$, o\`u $n=\rang_K(E)$, alors le degr\'e d'Arakelov de $\overline E$ est d\'efini comme 
\[\hdeg(\overline E)=\ln\frac{\mathrm{vol}(\phi(\mathbb B(\overline E)))}{\mathrm{vol}(\mathbb B(\overline K^n))},\]
o\`u $\mathbb B(.)$ d\'esigne la boule unit\'e ad\'elique, et $\mathrm{vol}$ d\'esigne une mesure de Haar sur l'espace ad\'elique $\mathbb A_K^n$. Cette d\'efinition ne d\'epend pas du choix de $\phi$ et $\mathrm{vol}$.} de $\overline E$ comme la diff\'erence entre la caract\'eristique d'Euler-Poincar\'e de $\overline E$ et celle  du fibr\'e vectorial ad\'elique trivial dont le rang est $\rang(E)$. Cela lui permet de g\'en\'eraliser la notion des pentes successives dans un cadre plus g\'en\'eral des fibr\'es vectoriels ad\'eliques non n\'ecessairement hermitiens. Comme dans le cas hermitien, les pentes successives de $\overline E$ sont d\'efinies comme les pentes du polygone de Harder-Narasimhan de $\overline E$, dont le graphe est le bord sup\'erieur de l'enveloppe convexe des points dans $\mathbb R^2$ de coordonn\'ees $(\rang(F),\ndeg(\overline F))$, o\`u $F$ parcourt l'ensemble des sous-espaces vectoriels de $E$. \`A l'aide de la m\'ethode d'ellipso\"{\i}de de John-L\"owner\footnote{On renvoie les lecteurs dans \cite[\S4]{Gaudron08} pour les d\'etails.}, on peut associer \`a $E$ une structure de fibr\'e vectoriel ad\'elique hermitien $(\|.\|_{v}')_{v\in M_K}$ de sorte que $\|.\|_{v}'=\|.\|_v$ si $v$ est une place non-archim\'edienne, et
\[\|.\|_v'\leqslant \|.\|_v\leqslant (\rang(E))^{1/2}\|.\|_{v}'\]
lorsque $v$ est archim\'edienne.
Si on note $\overline E{}'$ le fibr\'e vectoriel ad\'elique $(E,(\|.\|_v')_{v\in M_K})$, on a $\hmu_{i}(\overline E)\leqslant\hmu_i(\overline E{}')$ et $\Lambda_i(\overline E{}')\leqslant\Lambda_i(\overline E)+\frac 12\ln(\rang(E))$ pour tout $i\in\{1,\ldots,\rang(E)\}$. On obtient \`a partir de l'in\'egalit\'e \eqref{Equ:Siegeltronque} (appliqu\'ee \`a $\overline E{}'$) la relation suivante
\begin{equation}\label{Equ:Siegeltronquenonhem}\sum_{i=1}^{\rang(E)}\max(\hmu_i(\overline E),0)\leqslant\sum_{i=1}^{\rang(E)}\max(\Lambda_i(\overline E),0)+\rang(E)\ln(\rang(E)).\end{equation}

\section{Majoration de la fonction de Hilbert-Samuel arithm\'etique} \label{Sec:majroationarith}
Soit $K$ un corps de nombres. 
Dans ce paragraphe, on \'etablit un analogue arithm\'etique du corollaire \ref{Cor:majHS} pour un syst\`eme gradu\'e en fibr\'es vectoriels ad\'eliques sur $\Spec K$.  

Soit $X$ un sch\'ema projectif et int\`egre de dimension $d\geqslant 1$ sur $\Spec K$, $L$ un faisceau inversible gros sur $X$. On entend par \emph{syst\`eme lin\'eaire gradu\'e} de $L$ \emph{en fibr\'es vectoriels ad\'eliques} tout syst\`eme lin\'eaire gradu\'e $E_\sbullet=\bigoplus_{n\geqslant 0}E_n$ de $L$ dont chaque composante homog\`ene $E_n$ est muni d'une structure de fibr\'e vectoriel ad\'elique sur $\Spec K$ de telle sorte que 
\begin{equation}\label{Equ:suradditive}\|s\cdot s'\|_{v}\leqslant\|s\|_v\cdot\|s'\|_v\end{equation}
pour tout couple $(n,m)\in\mathbb N^2$ et tous $s\in E_{n,\mathbb C_v}$, $s'\in E_{m,\mathbb C_v}$. Cette in\'egalit\'e montre que la suite $(\Lambda_1(\overline E_n))_{n\geqslant 1}$ est sur-additive. Donc la suite $(\Lambda_1(\overline E_n)/n)_{n\geqslant 1}$ converge vers un \'el\'ement dans $\mathbb R\cup\{+\infty\}$ pourvu que $E_n\neq\{0\}$ pour tout entier $n$ suffisamment positif. On d\'esigne par $\widehat{\mu}_{\max}^{\mathrm{asy}}(\overline E_\sbullet)$ cette limite\footnote{Pour tout fibr\'e vectoriel ad\'elique non-nul (non n\'ecessairement hermitien) $\overline F$ sur $\Spec K$, on a $\Lambda_1(\overline F)\leqslant\hmu_{1}(\overline F)\leqslant\Lambda_1(\overline F)+ \ln(\rang(F))$. Par cons\'equent, $\hmu_{\max}^{\mathrm{asy}}(\overline E_\sbullet)$ est \'egal \`a $\displaystyle\lim_{n\rightarrow+\infty}\hmu_{\max}(\overline E_n)/n.$}. En outre, si $E_\sbullet$ contient un diviseur ample et si $\widehat{\mu}_{\max}^{\mathrm{asy}}(\overline E_\sbullet)<+\infty$, il est d\'emontr\'e dans \cite[th\'eor\`eme 2.8]{Boucksom_Chen} que la suite \[\frac{(d+1)!}{n^{d+1}}\sum_{i=1}^{\rang(E_n)}\max(\Lambda_i(\overline E_n),0),\qquad n\geqslant 1.\] converge vers un nombre r\'eel que l'on notera comme $\hnvol(\overline E_\sbullet)$.

\begin{theo}\label{Thm:HMarithmetique}
Soit $X$ un sch\'ema projectif et g\'eom\'etriquement int\`egre sur $\Spec K$,  $L$ un faisceau inversible gros sur $X$ et $\overline E_\sbullet=\bigoplus_{n\geqslant 0}\overline E_n$ un syst\`eme lin\'eaire gradu\'e de $L$ en fibr\'es vectoriels ad\'eliques. On suppose que le syst\`eme lin\'eaire gradu\'e $E_\sbullet$ contient un diviseur ample et que $X_{K^{\mathrm{a}}}$ poss\`ede un tour de fibrations sur courbes $\Theta$. Alors on a
\begin{equation}
\sum_{i=1}^{\rang(E_1)}\max(\Lambda_i(\overline E_1),0)\leqslant \hnvol(\overline E_\sbullet)+\hmu_{\max}^{\mathrm{asy}}(\overline E_\sbullet)\varepsilon^{\Theta}(E_{\sbullet,K^{\mathrm{a}}}).
\end{equation}
\end{theo}
\begin{proof}
Pour tout entier $n\in\mathbb N$ et tout nombre r\'eel $t$, on d\'esigne par $\mathcal F^tE_{n,K^{\mathrm{a}}}$ le sous-$K^{\mathrm{a}}$-espace vectoriel engendr\'e par les vecteurs non-nuls $s\in E_{n,K^{\mathrm{a}}}$ tels que $h_{\overline E_n}(s)\leqslant -t$. Il s'av\`ere que $(\mathcal F^tE_{n,K^{\mathrm{a}}})_{t\in\mathbb R}$ est une $\mathbb R$-filtration d\'ecroissante de $E_{n,K^{\mathrm{a}}}$. En outre, la relation \eqref{Equ:suradditive} montre que 
\[(\mathcal F^{t_1}E_{n_1,K^{\mathrm{a}}})\cdot(\mathcal F^{t_2}E_{n_2,K^{\mathrm{a}}})\subset\mathcal F^{t_1+t_2}E_{n_1+n_2,K^{\mathrm{a}}}\]
pour tous $(n_1,n_2)\in\mathbb N^2$ et $(t_1,t_2)\in\mathbb R^2$. Pour tout $t\in\mathbb R$, $E^t_{\sbullet}:=\bigoplus_{n\geqslant 0}\mathcal F^{nt}E_{n,K^{\mathrm{a}}}$ est alors un syst\`eme lin\'eaire gradu\'e de $L_{K^{\mathrm{a}}}$. D'apr\`es \cite[lemme 1.6]{Boucksom_Chen}, ce syst\`eme lin\'eaire gradu\'e contient un diviseur ample d\`es que $t<\hmu_{\max}^{\mathrm{asy}}(\overline E_\sbullet)$. En outre, on a (d'apr\`es le corollaire 1.13 du \emph{loc. cit.})
\[\hnvol(\overline E_\sbullet)=(d+1)\int_0^{+\infty}\vol(E^t_\sbullet)\,\mathrm{d}t=(d+1)\int_0^{\hmu_{\max}^{\mathrm{asy}}(\overline E_\sbullet)}\vol(E^t_\sbullet)\,\mathrm{d}t,\]
o\`u $d$ est la dimension de $X$. Le th\'eor\`eme \ref{Thm:majorationdeHS} appliqu\'e \`a $E_\sbullet^t$ montre que 
\[\rang(\mathcal F^tE_{1,K^{\mathrm{a}}})\leqslant\vol(E^t_{\sbullet,K^{\mathrm{a}}})+\varepsilon^{\Theta}(E_{\sbullet,K^{\mathrm{a}}}^t)\leqslant\vol(E^t_{\sbullet,K^{\mathrm{a}}})+\varepsilon^{\Theta}(E_{\sbullet,K^{\mathrm{a}}}).\]
On en d\'eduit
\[\begin{split}&\quad\;\sum_{i=1}^{\rang(E_1)}\max(\Lambda_i(\overline E_1),0)=\int_0^{\Lambda_1(\overline E_1)}\rang(\mathcal F^tE_{1,K^{\mathrm{a}}})\,\mathrm{d}t=\int_0^{\hmu_{\max}^{\mathrm{asy}}(\overline E_\sbullet)}\rang(\mathcal F^tE_{1,K^{\mathrm{a}}})\,\mathrm{d}t\\
&\leqslant \int_0^{\hmu_{\max}^{\mathrm{asy}}(\overline E_\sbullet)}\frac{\vol(E_\sbullet^t)}{d!}\,\mathrm{d}t+\hmu_{\max}^{\mathrm{asy}}(\overline E_\sbullet)\varepsilon^{\Theta}(E_{\sbullet,K^{\mathrm{a}}})=\frac{\hnvol(\overline E_\bullet)}{(d+1)!}+\hmu_{\max}^{\mathrm{asy}}(\overline E_\sbullet)\varepsilon^{\Theta}(E_{\sbullet,K^{\mathrm{a}}}),
\end{split}\]
o\`u la deuxi\`eme \'egalit\'e provient du fait que $\Lambda_1(\overline E_1)\leqslant\hmu_{\max}^{\mathrm{asy}}(\overline E_\sbullet)$. La d\'emonstration est donc achev\'ee.
\end{proof}

On d\'eduit du th\'eor\`eme pr\'ec\'edent et l'in\'egalit\'e \eqref{Equ:Siegeltronquenonhem} le r\'esultat suivant.

\begin{coro}\label{Cor:majrationdemu}
Avec les notations du th\'eor\`eme pr\'ec\'edent, on a 
\begin{equation}
\sum_{i=1}^{\rang(E_1)}\max(\hmu_i(\overline E_1),0)\leqslant\frac{\hnvol(\overline E_\sbullet)}{(d+1)!}+\hmu_{\max}^{\mathrm{asy}}(\overline E_\sbullet)\varepsilon^{\Theta}(E_{\sbullet,K^{\mathrm{a}}})+\rang(E_1)\ln(\rang(E_1)).
\end{equation}
\end{coro}

Soit $\pi:\mathscr X\rightarrow\Spec\mathcal O_K$ un morphisme projectif et plat d'un sch\'ema int\`egre $\mathscr X$ vers $\Spec\mathcal O_K$. Par faisceau inversible hermitien sur $\mathscr X$, on entend un $\mathcal O_{\mathscr X}$-module inversible $\mathscr L$ dont le tire en arri\`ere sur $\mathscr X^{\mathrm{an}}$ est muni d'une m\'etrique continue qui est invariante par la conjugaison complexe, o\`u $\mathscr X^{\mathrm{an}}$ d\'esigne l'espace analytique complexe associ\'e \`a $\mathscr X\otimes_{\mathbb Z}\mathbb C$. \'Etant donn\'e un faisceau inversible hermitien $\overline{\mathscr L}$ sur $\mathscr X$, on peut construire une structure de fibr\'e vectoriel ad\'elique sur $H^0(X,\mathscr L_K)$. En une place finie $\mathfrak p$,  la norme $\|.\|_{\mathfrak p}$ sur $H^0(X,\mathscr L_K)\otimes_K{\mathbb C_{\mathfrak p}}$ provient de la structure de $\mathcal O_K$-module de $\pi_*(\mathscr L)$~: la boule unit\'e ferm\'e pour la norme $\|.\|_{\mathfrak p}$ s'identifie \`a $\pi_*(\mathscr L)\otimes_{\mathcal O_K}\mathcal O_{\mathfrak p}$, o\`u $\mathcal O_{\mathfrak p}$ d\'esigne l'anneau de valuation de $\mathbb C_{\mathfrak p}$. En une place infinie $\sigma:K\rightarrow\mathbb C$, la norme $\|.\|_{\sigma}$ est la norme sup~: pour tout \'el\'ement $s\in H^0(X,\mathscr L_K)\otimes_{K,\sigma}\mathbb C$, on a
\[\|s\|_\sigma:=\sup_{x\in\mathscr X_\sigma(\mathbb C)}\|s(x)\|.\]
On utilise l'expression $\pi_*(\overline{\mathscr L})$ pour d\'esigner ce fibr\'e vectoriel ad\'elique. Ainsi \[\overline E_\sbullet=\bigoplus_{n\geqslant 0}\pi_*(\overline{\mathscr L}{}^{\otimes n})\] devient un syst\`eme lin\'eaire gradu\'e en fibr\'es vectoriels ad\'eliques sur $\Spec K$. Il s'av\`ere que le nombre $\hnvol(\overline E_\sbullet)$ d\'ecrit le comportement asymptotique du nombre de sections effectives de $\overline{\mathscr L}{}^{\otimes n}$ lorsque $n\rightarrow+\infty$. En effet, en utilisant l'in\'egalit\'e \eqref{Equ:h0degset} et la m\'ethode d'ellipso\"{\i}de de John-L\"owner, on peut montrer que
\[\hnvol(\overline{E}_\sbullet)=\frac{1}{[K:\mathbb Q]}\lim_{n\rightarrow+\infty}\frac{\widehat{h}^0(\overline E_n)}{n^{d+1}/(d+1)!},\]
o\`u $d$ est la dimension relative de $\pi$. Rappelons que la limite figurant dans le terme de droite de la formule est appel\'e le \emph{volume arithm\'etique}, not\'e comme $\hvol(\overline{\mathscr L})$. On renvoie les lecteurs dans \cite{Moriwaki07} o\`u cette notion a \'et\'e propos\'ee. Dans le cas o\`u le faisceau inversible hermitien $\overline{\mathscr L}$ est arithm\'etiquement nef\footnote{C'est-\`a-dire que $\mathscr L$ est nef relativement \`a $\pi$, la m\'etrique de $\overline{\mathscr L}$ est pluri-sous-harmonique, et la fonction hauteur sur l'ensemble des points alg\'ebriques de $\mathscr X_K$ d\'efinie par $\overline{\mathscr L}$ est \`a valeurs positives.}, le volume arithm\'etique de $\overline{\mathscr L}$ s'identifie au nombre d'auto-intersection arithm\'etique $\widehat{c}_1(\overline{\mathscr L})^{d+1}$. On d\'eduit alors du corollaire \ref{Cor:majrationdemu} le r\'esultat suivant.

\begin{theo}\label{Thm:HSarithmetic}
Soit $\pi:\mathscr X\rightarrow\Spec\mathcal O_K$ un morphisme projectif et plat. On suppose que la fibre g\'en\'erique g\'eom\'etrique $\mathscr X_{K^{\mathrm{a}}}$ est g\'eom\'etriquement int\`egre et admet un tour de fibrations sur courbes $\Theta$. Alors, pour tout faisceau inversible hermitien $\overline{\mathscr L}$ sur $\mathscr X$ qui est arithm\'etiquement nef et g\'en\'eriquement gros, il existe une fonction explicite $F_{\overline{\mathscr L}}:\mathbb N\rightarrow [0,+\infty[$ telle que \begin{equation}\label{Equ:estimationdeF}F_{\overline{\mathscr L}}(n)=[K:\mathbb Q]\frac{c_1(\mathscr L_K)^d}{(d-1)!}n^d\ln(n)+O(n^d),\quad n\rightarrow+\infty,\end{equation}
et que
\begin{equation}\label{Equ:estdeg}\hdeg(\pi_*(\overline{\mathscr L}{}^{\otimes n}))\leqslant\frac{\widehat{c}_1(\overline{\mathscr L})^{d+1}}{(d+1)!}n^{d+1}+F_{\overline{\mathscr L}}(n),\end{equation}
o\`u $d$ est la dimension relative de $\pi$.
\end{theo}
\begin{proof}
Pour tout entier $n\geqslant 1$, on note $r_n=\rang(\pi^*(\mathscr L^{\otimes n}))$. Comme $\mathscr L_K$ est nef et gros, on a (cf. \cite[corollaire 1.4.38]{LazarsfeldI})
\begin{equation}\label{Equ:estrn}r_n=\frac{c_1(\mathscr L_K)^d}{d!}n^d+O(n^{d-1}).\end{equation}
Soit $\overline E_{\sbullet}=\bigoplus_{n\geqslant 0}\pi_*(\overline{\mathscr L}{}^{\otimes n})$. Compte tenu de \cite[lemme 2.6]{Boucksom_Chen}, la relation $\hmu_{\max}^{\mathrm{asy}}(\overline E_\sbullet)<+\infty$ est satisfaite.  D'apr\`es le corollaire \ref{Cor:majrationdemu}, on obtient
\[\hdeg(\pi_*(\overline{\mathscr L}{}^{\otimes n}))\leqslant\sum_{i=1}^{r_n}\max(\hmu_i(\overline E_n),0)\leqslant [K:\mathbb Q]\hnvol(\overline E_\sbullet)\frac{n^{d+1}}{(d+1)!}+F_{\overline{\mathscr L}}(n)\]
avec
\[F_{\overline{\mathscr L}}(n)=[K:\mathbb Q]\big(n\widehat{
\mu}_{\max}^{\mathrm{asy}}(\overline{E}_\sbullet)\varepsilon^{\Theta}(E^{(n)}_{\sbullet,K^{\mathrm{a}}})+r_n\ln(r_n)\big)\] 
Comme $\varepsilon^{\Theta}(E_{\sbullet,K^{\mathrm{a}}}^{(n)})\leqslant n^{d-1}\varepsilon^{\Theta}(E_{\sbullet,K^{\mathrm{a}}})$, on d\'eduit de \eqref{Equ:estrn} la relation \eqref{Equ:estimationdeF}. Enfin, comme $\hvol(\overline{\mathscr L})=[K:\mathbb Q]\hnvol(\overline{E}_\sbullet)$, l'in\'egalit\'e \eqref{Equ:estdeg} est d\'emontr\'ee. 
\end{proof}

\begin{rema}
Si on remplace $\hdeg(.)$ par la somme des minima absolus successifs dans le th\'eor\`eme pr\'ec\'edent, on peut obtenir une majoration asymptotique o\`u le terme d'erreur est $O(n^d)$. Plus pr\'ecis\'ement, il existe une fonction explicite $\widetilde F_{\overline{\mathscr L}}:\mathbb N\rightarrow[0,+\infty[$ telle que $\widetilde F_{\overline{\mathscr L}}(n)=O(n^d)$ et que 
\[\sum_{i=1}^{r_n}\Lambda_i(\pi_*(\overline{\mathscr L}{}^{\otimes n}))\leqslant \frac{\widehat{c}_1(\overline{\mathscr L})^{d+1}}{(d+1)!}n^{d+1}+\widetilde F_{\overline{\mathscr L}}(n). \]
Cela sugg\`ere que le terme sous-dominant (de l'ordre $n^d\ln(n)$) dans le th\'eor\`eme de Riemann-Roch arithm\'etique (cf. \cite[\S2.2]{Soule92}) peut provenir du choix de m\'etrique et de la comparaison entre certains invariants arithm\'etiques de fibr\'e vectoriel hermitien sur $\Spec\mathcal O_K$. Des ph\'enom\`enes similaires existent aussi dans l'\'etude   des surfaces arithm\'etiques, comme par exemple \cite[th\'eor\`eme 3]{Faltings84} (voir aussi \cite[\S5.1]{Abbes-Bouche}).
\end{rema}

Dans la suite, on \'etablit un analogue du th\'eor\`eme \ref{Thm:HSarithmetic} pour la fonction $\widehat{h}^0$. En utilisant le deuxi\`eme th\'eor\`eme de Minkowski et la filtration par les minima (usuels), on peut majorer $\widehat{h}^0(\pi_*(\overline{\mathscr L}{}^{\otimes n}))$ par une fonction explicite. 

Soit $\overline M$ un fibr\'e vectoriel norm\'e sur $\Spec\mathbb Z$ (i.e. un r\'eseau dans un espace vectoriel norm\'e de dimension fini sur $\mathbb R$). Pour tout entier $i\in\{1,\ldots,n\}$, on d\'efinit le $i^{\text{\`eme}}$ \emph{minimum logarithmique} de $\overline M$ comme 
\[\lambda_i(\overline M):=\sup\Big\{t\in\mathbb R\,\Big|\,\rang_{\mathbb Q}\big(\mathrm{Vect}_{\mathbb Q}\{s\in M\,|\,\|s\|\leqslant\mathrm{e}^{-t}\}\big)\geqslant i\Big\}.\]
Rappelons que la caract\'eristique d'Euler-Poincar\'e de $\overline M$ est d\'efini comme
\[\chi(\overline M)=\ln\frac{\vol(B(M_{\mathbb R},\|.\|))}{\mathrm{covol}(M)},\]
o\`u $B(M_{\mathbb R},\|.\|)$ d\'esigne la boule unit\'e ferm\'ee dans $M_{\mathbb R}$, $\vol$ est une mesure de Haar sur $M_{\mathbb R}$ et $\mathrm{covol}(M)$ est la mesure de $M_{\mathbb R}/M$ par rapport \`a la mesure induite par $\vol$.

\begin{lemm}
Soit $\overline M$ un r\'eseau de rang $r>0$ dans un espace vectoriel norm\'e. On a
\begin{equation}
\widehat{h}^0(\overline M)\leqslant \sum_{i=1}^r\max(\lambda_i(\overline M),0)+r\ln(2)+\ln(2r!).
\end{equation}
\end{lemm}
\begin{proof}
Quitte \`a remplacer $M$ par le sous-r\'eseau engendr\'e par les \'el\'ements $s\in M$ v\'erifiant $\|s\|\leqslant 1$, on peut supposer que $\lambda_i(\overline M)\geqslant 0$ pour tout $i\in\{1,\ldots,r\}$. Rappelons que le deuxi\`eme th\'eor\`eme de Minkowski montre que
\[r\ln(2)-\ln(r!)\leqslant\chi(\overline M)-\sum_{i=1}^r\lambda_i(\overline M)\leqslant r\ln(2).\]
En outre, si on fixe une base de $M$ sur $\mathbb Z$ et identifie $M_{\mathbb R}$ \`a $\mathbb R^r$ via cette base, alors $B(M_{\mathbb R},\|.\|)$ est un corps convexe sym\'etrique dans $\mathbb R^r$ dont le volume est $\chi(\overline{M})$ (o\`u on a consid\'er\'e la mesure de Haar standard sur $\mathbb  R^r$). D'apr\`es un r\'esultat de Blichfeldt (cf. \cite[page 372]{Henze13}), on a 
\begin{equation}\label{Equ:Minkowski2}\widehat{h}^0(\overline M)\leqslant \ln(r!\exp(\chi(\overline M))+r)\leqslant \ln(2r!)+\chi(\overline M),\end{equation}
o\`u la deuxi\`eme in\'egalit\'e provient de l'hypoth\`ese $\forall\,i,\;\lambda_i(\overline M)\geqslant 0$. En effet, sous cette hypoth\`ese on a $r!\exp(\chi(\overline M))\geqslant 2^r> r$. D'apr\`es la deuxi\`eme in\'egalit\'e de \eqref{Equ:Minkowski2}, on obtient le r\'esultat.
\end{proof}

\begin{theo}\label{Thm:majorationh0}
Soit $\pi:\mathscr X\rightarrow\Spec\mathcal O_K$ un morphisme projectif et plat. On suppose que la fibre g\'en\'erique $\mathscr X_K$ est int\`egre. Alors, pour tout faisceau inversible hermitien $\overline{\mathscr L}$ sur $\mathscr X$ tel que $\mathscr L_K$ soit gros, il existe une fonction explicite $G_{\overline{\mathscr L}}:\mathbb N\rightarrow [0,+\infty[$ telle que \begin{equation}\label{Equ:estimationdeG}G_{\overline{\mathscr L}}(n)=[K:\mathbb Q]\frac{\vol(\mathscr L_K)}{(d-1)!}n^d\ln(n)+O(n^d),\quad n\rightarrow+\infty,\end{equation}
et que
\begin{equation}\label{Equ:estdegG}\widehat{h}^0(\pi_*(\overline{\mathscr L}{}^{\otimes n}))\leqslant\frac{\hvol(\overline{\mathscr L})}{(d+1)!}n^{d+1}+G_{\overline{\mathscr L}}(n),\end{equation}
o\`u $d$ est la dimension relative de $\pi$.
\end{theo}
\begin{proof}
Pour tout entier $n\geqslant 1$, on d\'esigne par $R_n$ le rang de $\pi_*(\mathscr L^{\otimes n})$ sur $\mathbb Z$. Comme $\mathscr L_K$ est gros, d'apr\`es le th\'eor\`eme \ref{Thm:majorationdeHS} on a
\begin{equation}
\label{Equ:estRn}
R_n=[K:\mathbb Q]\frac{\vol(\mathscr L_K)}{d!}n^d+o(n^{d-1}). 
\end{equation}
Soit $\overline E_\sbullet=\bigoplus_{n\geqslant 0}{\pi_*(\overline{\mathscr L}{}^{\otimes n})}$. On consid\`ere chaque $\overline E_n$ comme un r\'eseau dans l'espace vectoriel norm\'e $E_n\otimes_{\mathbb Q}\mathbb R$, qui s'identifie \`a 
\[\bigoplus_{v\in M_{K,\infty}}E_n\otimes_{K}K_v.\]
Si $\boldsymbol{s}=(s_v)_{v\in M_{K,\infty}}$ est un \'el\'ement de $E_n\otimes_{\mathbb Z}\mathbb R$, la norme de $\boldsymbol{s}$ est d\'efinie comme
\[\max_{v\in M_{K,\infty}}\|s_v\|_v.\]
On consid\`ere $\mathscr X_K$ comme un sch\'ema projectif sur $\Spec\mathbb Q$ (que l'on notera comme $\mathscr X_{\mathbb Q}$ dans la suite). De m\^eme, on consid\`ere $\mathscr L_K$ comme un faisceau inversible sur $\mathscr X_{\mathbb Q}$ (que l'on notera comme $\mathscr L_{\mathbb Q}$). Ainsi $E_\sbullet$ devient le syst\`eme lin\'eaire gradu\'e total de $\mathscr L_{\mathbb Q}$. Pour chaque entier $n\geqslant 0$, on munit $E_n$ (comme espace vectoriel sur $\mathbb Q$) de la $\mathbb R$-filtration $\mathcal F$ par les minima~:
\[\mathcal F^t(E_n)=\mathrm{Vect}_{\mathbb Q}\{s\in\pi_*(\mathscr L^{\otimes n})\,:\,\|s\|\leqslant \mathrm{e}^{-t}\}.\]  
Quitte \`a passer \`a une modification birationnelle (la quantit\'e $\widehat{h}^0(\pi_*(\overline{\mathscr L}{}^{\otimes n}))$ augment \'eventuellement, tandis que $\hvol(\overline{\mathscr L})$ reste inchang\'e), on peut supposer que $\mathscr X_{\mathbb Q}$ admet un tour de fibrations sur courbes $\Theta$. D'apr\`es le corollaire \ref{Cor:majHS} (appliqu\'e \`a $\overline{\mathscr L}{}^{\otimes n}$), on obtient
\[\sum_{i=1}^{R_n}\max(\lambda_i(\overline E_n),0)\leqslant\frac{\hvol(\overline{\mathscr L})}{(d+1)!}n^{d+1}+n^d\hmu_{\max}^{\mathrm{asy}}(\overline E_\sbullet)\varepsilon^{\Theta}(E_\sbullet),\]
o\`u on a utilis\'e les relations $\hmu_{\max}^{\mathrm{asy}}(\overline E_\sbullet^{(n)})=n\hmu_{\max}^{\mathrm{asy}}(\overline E_\sbullet)$ et $\varepsilon^{\Theta}(E_\sbullet^{(n)})\leqslant n^{d-1}\varepsilon^{\Theta}(E_\sbullet)$.
D'apr\`es le lemme pr\'ec\'edent, on obtient
\[\sum_{i=1}^{R_n}\max(\lambda_i(\overline E_n),0)\leqslant\frac{\hvol(\overline{\mathscr L})}{(d+1)!}n^{d+1}+G_{\overline{\mathscr L}}(n)\]
avec 
\[G_{\overline{\mathscr L}}(n)=n^d\hmu_{\max}^{\mathrm{asy}}(\overline E_\sbullet)\varepsilon^{\Theta}(E_\sbullet)+(R_n+1)\ln(2)+R_n\ln(R_n).\]
Enfin, la relation \eqref{Equ:estRn} montre aussit\^ot que
\[G_{\overline{\mathscr L}}(n)=[K:\mathbb Q]\frac{\vol(\mathscr L_K)}{(d-1)!}n^d\ln(n)+O(n^d).\]
Le th\'eor\`eme est donc achev\'e.
\end{proof}

\begin{rema}\label{Rem:comparaisonYuanZhang}
Dans le cas o\`u $\mathscr X$ est une surface arithm\'etique, il est int\'eressant de comparer le th\'eor\`eme pr\'ec\'edent \`a la majoration obtenu dans \cite{Yuan_Zhang13} (voir aussi l'analogue de ce travail dans le cadre de corps de fonctions \cite{Yuan_Zhang13b}). Asymptotiquement le terme $G_{\overline{\mathscr L}}(n)$ est meilleur que le terme d'erreur \begin{equation}\label{Equ:yuan-zhang}4n[K:\mathbb Q]\deg(\mathscr L_K)\ln(n[K:\mathbb Q]\deg(\mathscr L_K))\end{equation}
dans le th\'eor\`eme A du \emph{loc. cit.}. Cependant le terme \eqref{Equ:yuan-zhang} ne d\'epend que de l'information g\'eom\'etrique de la fibre g\'en\'erique $\mathscr L_K$. On se demande si une combinaison des deux m\'ethodes ne donne pas  une majoration effective de la somme des minima logarithmiques positifs de $\pi_*(\overline{\mathscr L}{}^n)$ dont le terme d'erreur est d'ordre $n^d$ et ne d\'epend que de la g\'eom\'etrie de $\mathscr L_K$ dans le cas o\`u $\mathscr X$ est une surface arithm\'etique.
\end{rema}

\section{Le cas de caract\'eristique positif}

Dans ce paragraphe, on \'etablit l'analogue du th\'eor\`eme \ref{Thm:majorationdeHS} dans le cas o\`u la caract\'eristique du corps de base est strictement positif. Soit $k$ un corps de caract\'eristique quelconque. La conclusion de la proposition \ref{Pro:produitdefiltrationHN}, qui est \'equivalente \`a la semi-stabilit\'e du produit tensoriel de tout couple de fibr\'es vectoriels semi-stables sur la courbe projective r\'eguli\`ere d\'efinie sur $k$, n'est cependant pas vrai en g\'en\'eral. On renvoie les lecteurs dans \cite{Gieseker73} pour un contre-exemple. La m\'ethode de $\mathbb R$-filtration de Harder-Narasimhan que l'on a d\'evelopp\'ee dans \S\,\ref{Sec:HS} n'est plus valable dans ce cadre-l\`a. On propose d'utiliser les minima successifs dans le cadre de corps de fonction pour surmonter cette difficult\'e.  

Soient $C$ une courbe projective et r\'eguli\`ere sur $\Spec k$. On d\'esigne par $K$ le corps des fonctions rationnelles sur $C$. Si $E$ est un fibr\'e vectoriel sur $C$, on d\'esigne par $\mathcal O_E(1)$ le faisceau inversible universel du sch\'ema $\mathbb P(E)$. On peut d\'efinir une fonction de hauteur sur l'ensemble des $k$-points de $\mathbb P(E)$ \`a valeurs dans $K$ comme la suite. Si $x$ est un point dans $\mathbb P(E)_k(K)$, il se prolonge en une section $\mathscr P_x:C\rightarrow\mathbb P(E)$ de $\mathbb P(E)$. On d\'efinit la hauteur $x$ comme
\[h_E(x):=\deg(\mathscr P_x^*\mathcal O_E(1)).\]
La fonction de hauteur nous permet de d\'efinir une filtration $\mathcal F$ sur l'espace $K$-vectoriel $E_{K}$ (la fibre g\'en\'erique de $E$) comme la suite\footnote{Rappelons qu'un \'el\'ement dans $\mathbb P(E^\vee)_k(K)$ correspond \`a un sous-espace $K$-vectoriel de rang $1$ de $E_{K}$.} 
\begin{equation}\label{Equ:filtrationparminima}\mathcal F^t(E_{K}):=\mathrm{Vect}_{K}\big\{x\in\mathbb P(E^\vee)_k(K)\,|\,h_{E^{\vee}}(x)\leqslant -t\big\}.\end{equation} 
Pour tout entier $i\in\{1,\ldots,\rang(E)\}$, on d\'esigne par $\lambda_i(E)$ le plus grand nombre r\'eel $t$ tel que $\rang(\mathcal F^t(E_{K}))\geqslant i$. Les nombres $\lambda_i(E)$ et la filtration $\mathcal F$ sont reli\'ees par la formule suivante~:
\begin{equation}\label{Equ:lambdaint}
\sum_{i=1}^{\rang(E)}\max(\lambda_i(E),0)=\int_0^{+\infty}\rang_K(\mathcal F^t(E_K))\,\mathrm{d}t=\int_0^{\lambda_1(E)}\rang_K(\mathcal F^t(E_K))\,\mathrm{d}t.
\end{equation} 

Les quantit\'es $\lambda_i(E) $ devraient \^etre consid\'er\'ees comme l'analogue des minima successifs dans le cadre de corps de fonction. On d\'esigne par $M_K$ l'ensemble des points ferm\'es dans la courbe $C$, consid\'er\'e comme l'ensemble des places du corps de fonctions $K$. Pour tout point $v\in M_K$, on d\'esigne par $|.|_v$ la valeur absolue sur $K$ d\'efinie comme
\[\forall\,f\in K^{\times}\,\quad |f|_v=\exp(-[k(v):k]\mathrm{ord}_v(f)),\]
o\`u $k(v)$ est le corps r\'esiduel de $v$ et $\mathrm{ord}_v(f)$ d\'esigne l'ordre d'annulation de $f$ en $v$. Comme dans le cas de corps de nombres, on d\'esigne par $K_v$ le compl\'et\'e de $K$ par rapport \`a cette valeur absolue et $\mathbb C_v$ le compl\'et\'e d'une cl\^oture alg\'ebrique de $K_v$, sur lequel la valeur absolue $|.|_v$ s'\'etend de fa\c{c}on unique. On d\'esigne par $\mathcal O_v$ l'anneau de valuation de $\mathbb C_v$ par rapport \`a cette valeur absolue (qui est non-archim\'edienne).

Si $E$ est un fibr\'e vectoriel sur $C$, alors sa structure de $\mathcal O_C$-module d\'efinit, pour chaque place $v\in M_K$, une norme $\|.\|_{E,v}$ sur $E\otimes_{\mathcal O_C}\mathbb C_v$ dont la boule unit\'e ferm\'ee est $E\otimes_{\mathcal O_C}\mathcal O_v$. Cette norme est invariante par l'action du groupe de Galois $\mathrm{Gal}(\mathbb C_v/K_v)$. On obtient alors un fibr\'e vectoriel ad\'elique sur $\Spec K$ au sens de \cite[\S3]{Gaudron08}, et il s'av\`ere que l'on peut exprimer le degr\'e de $E$ sous la forme
\[\deg(E)=-\sum_{v\in M_K}\ln\|s_1\wedge\cdots\wedge s_r\|_{E,v},\]
o\`u $(s_1,\ldots,s_r)$ est une base quelconque de $E_K$. En outre, les nombres $\lambda_i(E)$ sont pr\'ecis\'ement les minima successifs logarithmiques suivant Thunder \cite{Thunder96} dans le cadre de corps de fonctions.

D'apr\`es un r\'esultat de Roy et Thunder \cite[th\'eor\`eme 2.1]{Roy_Thunder96}~: on a 
\begin{equation}\label{Equ:Siegelfun}\sum_{i=1}^{\rang(E)}\lambda_i(E)\geqslant\deg(E)-\rang(E)\ell(g(C)),\end{equation}
o\`u $g(C)$ le genre de $C$ et $\ell$ est une fonction affine qui ne d\'epend que du degr\'e effectif du corps de fonction $K$. On renvoie les lecteur dans \cite[page 5]{Roy_Thunder96} pour la forme explicite de cette fonction.
On en d\'eduit le r\'esultat suivant.
\begin{prop}
Soit $C$ une courbe projective r\'eguli\`ere d\'efinie sur un corps $k$. Si $E$ est un fibr\'e vectoriel sur $C$, on a
\begin{equation}\label{Equ:majdedeg+l}
\deg_+(E)=\sum_{i=1}^{\rang(E)}\max({\mu}_i(E),0)\leqslant\sum_{i=1}^{\rang(E)}\max(\lambda_i(E),0)+\rang(E)\ell(g(C)).
\end{equation}
\end{prop}
\begin{proof}
Quitte \`a remplacer $E$ par  le dernier sous-fibr\'e vectoriel de pente minimale positive dans le drapeau de Harder-Narasimhan de $E$, on peut supposer que $\mu_{\min}(E)\geqslant 0$. Dans ce cas-l\`a on a \[\deg_+(E)=\deg(E)=\sum_{i=1}^{\rang(E)}\lambda_i(E)+\rang(E)\ell(g(C))\leqslant \sum_{i=1}^{\rang(E)}\max(\lambda_i(E),0)+\rang(E)\ell(g(C)),\]
o\`u la deuxi\`eme \'egalit\'e provient de \eqref{Equ:Siegelfun}. Le r\'esultat est donc d\'emontr\'e. 
\end{proof} 

Soient $E$ et $F$ deux fibr\'es vectoriels sur $C$. Si $x$ et $y$ sont respectivement deux $k$-points de $\mathbb P(E^\vee)$ et $\mathbb P(F^\vee)$ \`a valeurs dans $K$, alors $x\otimes y$ (vu comme un sous-espace vectoriel de rang un de $E_K\otimes F_K$) est un $k$-point de $\mathbb P(E^\vee\otimes F^\vee)$ \`a valeurs dans $K$ qui v\'erifie la relation suivante
\[h_{E^\vee\otimes F^\vee}(x\otimes y)=h_{E^\vee}(x)+h_{F^\vee}(y).\]
On obtient donc le r\'esultat suivant~:
\begin{prop}
Soit $E_\sbullet=\bigoplus_{n\geqslant 0}E_n$ un $\mathcal O_C$-alg\`ebre gradu\'ee. On suppose que chaque $E_n$ est un fibr\'e vectoriel sur $C$. Pour tout $(a,b)\in\mathbb R^2$ et tout $(n,m)\in\mathbb N^2$, la relation suivante est v\'erifi\'ee~:
\begin{equation}\label{Equ:filtree}
\mathcal F^a(E_{n,K})\mathcal F^b(E_{m,K})\subset\mathcal F^{a+b}(E_{n+m,K}),
\end{equation}
o\`u $\mathcal F$ d\'esigne la $\mathbb R$-filtration par minima d\'efinie dans \eqref{Equ:filtrationparminima}.
\end{prop}

\begin{rema}\label{Rem:comparaisonlambda1}
 On fixe un faisceau inversible ample $M$ sur $C$. Soit $a$ le degr\'e de $M$.
Si $E$ est un fibr\'e vectoriel non-nul sur $C$, les in\'egalit\'es \[\lambda_1(E)\leqslant\mu_{\max}(E)\leqslant\lambda_1(E)+g-1+a\]
sont toujours v\'erifi\'ees. La premi\`ere in\'egalit\'e est triviale. Pour la deuxi\`eme in\'egalit\'e, on peut utiliser l'invariance de la quantit\'e $\mu_{\max}(E)-\lambda_1(E)$ par le produit tensoriel d'un $\mathcal O_C$-module inversible. Quitte \`a remplacer $E$ par le produit tensoriel de $E$ avec une puissance tensorielle (\'eventuellement d'exposant n\'egatif) du faisceau inversible $M$, on peut supposer $g-1<\mu_{\max}(E)\leqslant g-1+a$. D'apr\`es le th\'eor\`eme de Riemann-Roch, on obtient que $\lambda_1(E)\geqslant 0$ (cf. \cite[lemme 2.1]{Chen_pm}). On obtient donc $\mu_{\max}(E)-\lambda_1(E)\leqslant g-1+a$. Cette in\'egalit\'e montre que, si $E_\sbullet$ est une $\mathcal O_C$-alg\`ebre gradu\'ee en fibr\'es vectoriels sur $C$ telle que $E_n$ soit non-nul pour $n$ assez grand, alors on a 
\[\lim_{n\rightarrow+\infty}\frac{\mu_{\max}(E_n)}{n}=\lim_{n\rightarrow+\infty}\frac{\lambda_1(E_n)}{n}\in\mathbb R\cup\{+\infty\}.\]
En outre, la relation \eqref{Equ:filtree} montre que la suite $\lambda_1(E_n)$ est sur-additive. On obtient donc
\begin{equation}
\forall\,p\geqslant 1,\quad
\lambda_1(E_p)\leqslant p\lim_{n\rightarrow+\infty}\frac{\mu_{\max}(E_n)}{n}.
\end{equation}
\end{rema}

La proposition pr\'ec\'edente nous permet de retrouver les r\'esultats pr\'esent\'es dans \S\,\ref{Sec:HS}, quitte \`a remplacer la filtration de Harder-Narasimhan par la filtration par minima. Soit $X$ un sch\'ema projectif et int\`egre sur $\Spec k$, muni d'un tour de fibrations sur courbes $\Theta$. On suppose que la dimension de Krull de $X$ est $d+1$. 
Soient $L$ un faisceau inversible gros sur $X$ et $V_\sbullet$ un syst\`eme lin\'eaire gradu\'e de $L$, qui contient un diviseur ample.
Si $d=0$, on d\'efinit  
\[\widetilde{\varepsilon}^{\Theta}(V_\sbullet):=\max(g(\Theta)-1,1),\]
o\`u $g(\Theta)$ est le genre du tour $\Theta$ d\'efini dans \S\,\ref{Sec:tourdefibration}. Si $d\geqslant 1$, on d\'efinit $\widetilde{\varepsilon}^{\Theta}(V_\sbullet)$ de fa\c{c}on r\'ecursive comme la suite. On suppose que $\Theta$ est de la forme $(p_0:X\rightarrow C_0,\Theta')$, o\`u $\Theta'$ est un tour de fibrations sur courbes de la fibre g\'en\'erique de $p_0$. Soient en outre $g_0$ le genre de la courbe $C_0$,  $\mu_0:=\mu^{p_0}_{\max}(V_\sbullet)$ (cf. la d\'efinition \ref{Def:imagedirecte}), et $W_\sbullet$ la fibre g\'en\'erique de $p_{0*}(V_\sbullet)$. On d\'efinit
\[\widetilde{\varepsilon}^{\Theta}(V_\sbullet)=\mu_0\widetilde{\varepsilon}^{\Theta'}(W_\sbullet)+\Big(\frac{\vol(W_\sbullet)}{d!}+\widetilde\varepsilon^{\Theta'}(W_\sbullet)\Big)(\max(g_0-1,1)+\ell(g_0)).\]
On voit aussit\^ot que $\widetilde{\varepsilon}^{\Theta}(V_\sbullet)\leqslant\widetilde{\varepsilon}^{\Theta}(V_\sbullet')$ si $V_\sbullet$ est contenu dans un autre syst\`eme lin\'eaire gradu\'e $V_\sbullet'$ d'un autre faisceau inversible gros $L'$. En outre, par r\'ecurrence sur $d$ on peut v\'erifier que $\widetilde{\varepsilon}^{\Theta}(V_\sbullet^{(p)})\leqslant p^d\widetilde{\varepsilon}^{\Theta}(V_\sbullet)$ pour tout entier $p\geqslant 1$.

\begin{theo}\label{Thm:majorationdeHSbis}
Soit $X$ un sch\'ema projectif et int\`egre sur $\Spec k$ muni d'un tour de fibrations sur courbes $\Theta=(p_i:X_i\rightarrow C_i)_{i=0}^d$, et $L$ un faisceau inversible sur $X$. 
Si $V_\sbullet=\bigoplus_{n\geqslant 0}V_n$ est un syst\`eme lin\'eaire gradu\'e de $L$, qui contient un diviseur ample, alors on a 
\begin{equation}\label{Equ:majorationbis}\rang_k(V_1)\leqslant\frac{\mathrm{vol}(V_\sbullet)}{(d+1)!}+\widetilde{\varepsilon}^{\Theta}(V_\sbullet). \end{equation}
\end{theo}
\begin{proof}
La d\'emonstration est presque identique \`a celle du th\'eor\`eme \ref{Thm:majorationdeHS}. Il suffit de remplacer les filtrations de Harder-Narasimhan par les filtrations par minima. Le cas o\`u $d=0$ utilise notamment les estim\'es d\'emontr\'ees dans le th\'eor\`eme \ref{Thm:h0etdegplus} qui sont valables pour tout corps $k$, et la d\'emonstration reste donc inchang\'ee.

Dans la suite, on suppose $d\geqslant 1$. On suppose en outre que $\Theta$ est de la forme $(p_0:X\rightarrow C_0,\Theta')$, o\`u $\Theta'$ est un tour de fibrations sur courbes de la fibre g\'en\'erique de $p_0$. Soient $g_0$ le genre de la courbe $C_0$,  $\mu_0:=\mu^{p_0}_{\max}(V_\sbullet)$, et $W_\sbullet$ la fibre g\'en\'erique de $E_\sbullet=p_{0*}(V_\sbullet)$. On munit $W_\sbullet$ de la filtration par minima $\mathcal F$ et on note $W_\sbullet^t=\bigoplus_{n\geqslant 0}\mathcal F^{nt}W_n$. On a encore
\[\vol(V_\sbullet)=\vol(E_\sbullet)=(d+1)\int_0^{\mu_0}\rang(W_\sbullet^t)\,\mathrm{d}t,\]
o\`u la premi\`ere \'egalit\'e provient du th\'eor\`eme \ref{Thm:passageauxfibres} et on a utilis\'e la relation $\lambda_1(E_1)\leqslant\mu_0$ (cf. la remarque \ref{Rem:comparaisonlambda1}) dans la deuxi\`eme \'egalit\'e. 
En outre, les relations \eqref{Equ:majdedeg+l} et \eqref{Equ:lambdaint} montre que
\[\deg_+(E_1)\leqslant\int_0^{\mu_0}\rang(W_1^t)\,\mathrm{d}t+\rang(W_1)\ell(g_0).\]
On applique l'hypoth\`ese de r\'ecurrence \`a $W_\sbullet^t$ et obtient 
\[\begin{split}\deg_+(E_1)&\leqslant\int_0^{\mu_0}\Big(
\frac{\vol(W_1^t)}{d!}+\widetilde{\varepsilon}^{\Theta'}(W_\sbullet^t)\Big)\,\mathrm{d}t+\rang(W_1)\ell(g_0)\\
&\leqslant\frac{\vol(V_\sbullet)}{(d+1)!}+\mu_0\widetilde{\varepsilon}^{\Theta'}(W_\sbullet^t)+\rang(W_1)\ell(g_0).
\end{split}\]
On en d\'eduit 
\[\begin{split}\rang(V_1)&\leqslant h^0(E_1)\leqslant\deg_+(E_1)+\rang(W_1)\max(g_0-1,1)\\
&\leqslant\frac{\vol(V_\sbullet)}{(d+1)!}+\mu_0\widetilde{\varepsilon}^{\Theta'}(W_\sbullet^t)+\rang(W_1)(\ell(g_0)+\max(g_0-1,1)).
\end{split}\]
On applique alors l'hypoth\`ese de r\'ecurrence \`a $W_\sbullet$ et obtient le r\'esultat souhait\'e.
\end{proof}

\backmatter
\bibliography{chen}

\def\cprime{$'$} \def\cprime{$'$}
\providecommand{\bysame}{\leavevmode ---\ }
\providecommand{\og}{``}
\providecommand{\fg}{''}
\providecommand{\smfandname}{\&}
\providecommand{\smfedsname}{\'eds.}
\providecommand{\smfedname}{\'ed.}
\providecommand{\smfmastersthesisname}{M\'emoire}
\providecommand{\smfphdthesisname}{Th\`ese}
\begin{thebibliography}{10}

\bibitem{Abbes-Bouche}
{\scshape A.~Abbes {\normalfont \smfandname} T.~Bouche} -- {\og Th\'eor\`eme de
  {H}ilbert-{S}amuel ``arithm\'etique''\fg}, \emph{Universit\'e de Grenoble.
  Annales de l'Institut Fourier} \textbf{45} (1995), no.~2, p.~375--401.

\bibitem{Banaszczyk95}
{\scshape W.~Banaszczyk} -- {\og Inequalities for convex bodies and polar
  reciprocal lattics in {${\bf R}\sp n$}\fg}, \emph{Discrete \& Computational
  Geometry. An International Journal of Mathematics and Computer Science}
  \textbf{13} (1995), no.~2, p.~217--231.

\bibitem{Betke_Boroczky99}
{\scshape U.~Betke {\normalfont \smfandname} K.~B{\"o}r{\"o}czky, Jr.} -- {\og
  Asymptotic formulae for the lattice point enumerator\fg}, \emph{Canadian
  Journal of Mathematics. Journal Canadien de Math\'ematiques} \textbf{51}
  (1999), no.~2, p.~225--249.

\bibitem{Blichfeld14}
{\scshape H.~F. Blichfeld} -- {\og Notes on geometry of numbers. {A}nnouncement
  to the {O}ctober meeting of the {S}an {F}rancisco {S}ection\fg},
  \emph{Bulletin of the American Mathematical Society} \textbf{27} (1921),
  no.~4, p.~152--153.

\bibitem{Bombieri_Vaaler83}
{\scshape E.~Bombieri {\normalfont \smfandname} J.~Vaaler} -- {\og On
  {S}iegel's lemma\fg}, \emph{Inventiones Mathematicae} \textbf{73} (1983),
  no.~1, p.~11--32.

\bibitem{Bost_Chen}
{\scshape J.-B. Bost {\normalfont \smfandname} H.~Chen} -- {\og Concerning the
  semistability of tensor products in {A}rakelov geometry\fg}, \emph{Journal de
  Math\'ematiques Pures et Appliqu\'ees. Neuvi\`eme S\'erie} \textbf{99}
  (2013), no.~4, p.~436--488.

\bibitem{Boucksom_Chen}
{\scshape S.~Boucksom {\normalfont \smfandname} H.~Chen} -- {\og Okounkov
  bodies of filtered linear series\fg}, \emph{Compositio Mathematica}
  \textbf{147} (2011), no.~4, p.~1205--1229.

\bibitem{Chardin89}
{\scshape M.~Chardin} -- {\og Une majoration de la fonction de {H}ilbert et ses
  cons\'equences pour l'interpolation alg\'ebrique\fg}, \emph{Bulletin de la
  Soci\'et\'e Math\'ematique de France} \textbf{117} (1989), no.~3,
  p.~305--318.

\bibitem{Chen08}
{\scshape H.~Chen} -- {\og {P}ositive degree and arithmetic bigness\fg}, 2008,
  \url{arXiv:0803.2583}.

\bibitem{Chen_pm}
\bysame , {\og Maximal slope of tensor product of {H}ermitian vector
  bundles\fg}, \emph{Journal of algebraic geometry} \textbf{18} (2009), no.~3,
  p.~575--603.

\bibitem{Chen10}
\bysame , {\og Arithmetic {F}ujita approximation\fg}, \emph{Annales
  Scientifiques de l'\'Ecole Normale Sup\'erieure. Quatri\`eme S\'erie}
  \textbf{43} (2010), no.~4, p.~555--578.

\bibitem{Chen10b}
\bysame , {\og Convergence des polygones de {H}arder-{N}arasimhan\fg},
  \emph{{M}\'emoires de la {S}oci\'et\'e Math\'ematique de France} \textbf{120}
  (2010), p.~1--120.

\bibitem{Faltings84}
{\scshape G.~Faltings} -- {\og Calculus on arithmetic surfaces\fg},
  \emph{Annals of Mathematics. Second Series} \textbf{119} (1984), no.~2,
  p.~387--424.

\bibitem{Fujita94}
{\scshape T.~Fujita} -- {\og Approximating {Z}ariski decomposition of big line
  bundles\fg}, \emph{Kodai Mathematical Journal} \textbf{17} (1994), no.~1,
  p.~1--3.

\bibitem{Gaudron08}
{\scshape {\'E}.~Gaudron} -- {\og Pentes de fibr\'es vectoriels ad\'eliques sur
  un corps globale\fg}, \emph{Rendiconti del Seminario Matematico della
  Universit\`a di Padova} \textbf{119} (2008), p.~21--95.

\bibitem{Gaudron09}
\bysame , {\og G\'eom\'etrie des nombres ad\'elique et lemmes de {S}iegel
  g\'en\'eralis\'es\fg}, \emph{Manuscripta Mathematica} \textbf{130} (2009),
  no.~2, p.~159--182.

\bibitem{Gaudron_Remond13}
{\scshape {\'E}.~Gaudron {\normalfont \smfandname} G.~R\'emond} -- {\og Corps
  de siegel\fg}, pr\'epublication, 2013.

\bibitem{Gaudron_Remond}
{\scshape {\'E}.~Gaudron {\normalfont \smfandname} G.~R{\'e}mond} -- {\og
  Minima, pentes et alg\`ebre tensorielle\fg}, \emph{Israel Journal of
  Mathematics} \textbf{195} (2013), no.~2, p.~565--591.

\bibitem{Gieseker73}
{\scshape D.~Gieseker} -- {\og Stable vector bundles and the {F}robenius
  morphism\fg}, \emph{Annales Scientifiques de l'\'Ecole Normale Sup\'erieure.
  Quatri\`eme S\'erie} \textbf{6} (1973), p.~95--101.

\bibitem{Gillet-Soule91}
{\scshape H.~Gillet {\normalfont \smfandname} C.~Soul{\'e}} -- {\og On the
  number of lattice points in convex symmetric bodies and their duals\fg},
  \emph{Israel Journal of Mathematics} \textbf{74} (1991), no.~2-3,
  p.~347--357.

\bibitem{EGAIII_1}
{\scshape A.~Grothendieck {\normalfont \smfandname} J.~Dieudonn\'e} -- {\og
  \'{E}l\'ements de g\'eom\'etrie alg\'ebrique. {III}. \'{E}tude cohomologique
  des faisceaux coh\'erents. {I}\fg}, \emph{Institut des Hautes \'Etudes
  Scientifiques. Publications Math\'ematiques} (1961), no.~11, p.~167.

\bibitem{Henk02}
{\scshape M.~Henk} -- {\og Successive minima and lattice points\fg},
  \emph{Rendiconti del Circolo Matematico di Palermo. Serie II. Supplemento}
  (2002), no.~70, part I, p.~377--384, IV International Conference in
  ``Stochastic Geometry, Convex Bodies, Empirical Measures $\&$ Applications to
  Engineering Science'', Vol. I (Tropea, 2001).

\bibitem{Henze13}
{\scshape M.~Henze} -- {\og A {B}lichfeldt-type inequality for centrally
  symmetric convex bodies\fg}, \emph{Monatshefte f\"ur Mathematik} \textbf{170}
  (2013), no.~3-4, p.~371--379.

\bibitem{Huyb}
{\scshape D.~Huybrechts {\normalfont \smfandname} M.~Lehn} -- \emph{The
  geometry of moduli spaces of sheaves}, Aspects of Mathematics, E31, Friedr.
  Vieweg \& Sohn, Braunschweig, 1997.

\bibitem{Kaveh_Khovanskii}
{\scshape K.~Kaveh {\normalfont \smfandname} A.~Khovanskii} -- {\og Algebraic
  equations and convex bodies\fg}, in \emph{Perspectives in analysis, geometry,
  and topology}, Progr. Math., vol. 296, Birkh\"auser/Springer, New York, 2012,
  p.~263--282.

\bibitem{Kollar_Matsusaka}
{\scshape J.~Koll{\'a}r {\normalfont \smfandname} T.~Matsusaka} -- {\og
  Riemann-{R}och type inequalities\fg}, \emph{Amer. J. Math.} \textbf{105}
  (1983), no.~1, p.~229--252.

\bibitem{LazarsfeldI}
{\scshape R.~Lazarsfeld} -- \emph{Positivity in algebraic geometry. {I}},
  Ergebnisse der Mathematik und ihrer Grenzgebiete. 3. Folge. A Series of
  Modern Surveys in Mathematics, vol.~48, Springer-Verlag, Berlin, 2004,
  Classical setting: line bundles and linear series.

\bibitem{Lazarsfeld_Mustata08}
{\scshape R.~Lazarsfeld {\normalfont \smfandname} M.~Musta{\c{t}}{\u{a}}} --
  {\og Convex bodies associated to linear series\fg}, \emph{Annales
  Scientifiques de l'\'Ecole Normale Sup\'erieure. Quatri\`eme S\'erie}
  \textbf{42} (2009), no.~5, p.~783--835.

\bibitem{LiuQing}
{\scshape Q.~Liu} -- \emph{Algebraic geometry and arithmetic curves}, Oxford
  Graduate Texts in Mathematics, vol.~6, Oxford University Press, Oxford, 2002,
  Translated from the French by Reinie Ern{\'e}, Oxford Science Publications.

\bibitem{Moriwaki07}
{\scshape A.~Moriwaki} -- {\og Continuity of volumes on arithmetic
  varieties\fg}, \emph{Journal of algebraic geometry} \textbf{18} (2009),
  no.~3, p.~407--457.

\bibitem{Nara_Se65}
{\scshape M.~S. Narasimhan {\normalfont \smfandname} C.~S. Seshadri} -- {\og
  Stable and unitary vector bundles on a compact {R}iemann surface\fg},
  \emph{Annals of Mathematics. Second Series} \textbf{82} (1965), p.~540--567.

\bibitem{LNM1752}
{\scshape Y.~Nesterenko {\normalfont \smfandname} P.~Philippon} (\smfedsname)
  -- \emph{Introduction to algebraic independence theory}, Lecture Notes in
  Mathematics, vol. 1752, Springer-Verlag, Berlin, 2001.

\bibitem{Nesterenko85}
{\scshape Y.~V. Nesterenko} -- {\og Estimates for the characteristic function
  of a prime ideal\fg}, \emph{Mathematics of the USSR-Sbornik} \textbf{51}
  (1985), no.~1, p.~9--32.

\bibitem{Roy_Thunder96}
{\scshape D.~Roy {\normalfont \smfandname} J.~L. Thunder} -- {\og An absolute
  {S}iegel's lemma\fg}, \emph{Journal f\"ur die Reine und Angewandte
  Mathematik} \textbf{476} (1996), p.~1--26.

\bibitem{Sombra97}
{\scshape M.~Sombra} -- {\og Bounds for the {H}ilbert function of polynomial
  ideals and for the degrees in the {N}ullstellensatz\fg}, \emph{Journal of
  Pure and Applied Algebra} \textbf{117/118} (1997), p.~565--599, Algorithms
  for algebra (Eindhoven, 1996).

\bibitem{Soule92}
{\scshape C.~Soul\'e, D.~Abramovich, J.-F. Burnol {\normalfont \smfandname}
  J.~Kramer} -- \emph{Lectures on arakelov geometry}, Cambridge studies in
  advanced mathematics, vol.~33, Cambridge University Press, 1992.

\bibitem{Takagi07}
{\scshape S.~Takagi} -- {\og Fujita's approximation theorem in positive
  characteristics\fg}, \emph{Journal of Mathematics of Kyoto University}
  \textbf{47} (2007), no.~1, p.~179--202.

\bibitem{Thunder96}
{\scshape J.~L. Thunder} -- {\og An adelic {M}inkowski-{H}lawka theorem and an
  application to {S}iegel's lemma\fg}, \emph{Journal f\"ur die Reine und
  Angewandte Mathematik} \textbf{475} (1996), p.~167--185.

\bibitem{Yuan09}
{\scshape X.~Yuan} -- {\og On volumes of arithmetic line bundles\fg},
  \emph{Compositio Mathematica} \textbf{145} (2009), no.~6, p.~1447--1464.

\bibitem{Yuan_Zhang13}
{\scshape X.~Yuan {\normalfont \smfandname} T.~Zhang} -- {\og Effective bound
  of linear series on arithmetic surfaces\fg}, \emph{Duke Mathematical Journal}
  \textbf{162} (2013), no.~10, p.~1723--1770.

\bibitem{Yuan_Zhang13b}
\bysame , {\og Relative noether inequality on fibered surfaces\fg},
  Pr\'epublication, 2013.

\bibitem{Yuan_Zhang14}
\bysame , {\og Effective bounds of linear series on algebraic varieties and
  arithmetic varieties\fg}, Pr\'epublication, 2014.

\bibitem{Zhang95}
{\scshape S.~Zhang} -- {\og Positive line bundles on arithmetic varieties\fg},
  \emph{{J}ournal of the {A}merican {M}athematical {S}ociety} \textbf{8}
  (1995), no.~1, p.~187--221.

\end{thebibliography}
\bibliographystyle{smfplain}

\end{document}